\documentclass[11pt]{article}

\usepackage[margin=1in]{geometry}
\usepackage[utf8]{inputenc}
\usepackage{listings}
\usepackage[framed , numbered]{matlab-prettifier}
\usepackage{amsmath,enumerate,mathtools,amsthm,amssymb,enumerate,multirow,multicol,setspace,xcolor}

\usepackage[normalem]{ulem}
\usepackage{comment}
\usepackage{float}
\floatstyle{plaintop}
\restylefloat{table}
\restylefloat{table}
\usepackage[caption=false]{subfig}
\usepackage{graphicx,epstopdf}
\usepackage{xcolor, soul}
\sethlcolor{yellow}
\usepackage{graphicx}

\usepackage{footnote}
\makesavenoteenv{align}
\makesavenoteenv{tabularx}
\makesavenoteenv{table}
\usepackage{algorithm,caption}
\usepackage{algorithmic}

\usepackage{cleveref}

\allowdisplaybreaks

\makeatletter
\newcommand*\bigcdot{\mathpalette\bigcdot@{.5}}
\newcommand*\bigcdot@[2]{\mathbin{\vcenter{\hbox{\scalebox{#2}{$\m@th#1\bullet$}}}}}
\makeatother

\DeclareMathOperator*{\argmax}{arg\,max}
\DeclareMathOperator*{\argmin}{arg\,min}

\newcommand{\N}{\mathbb{N}}
\newcommand{\R}{\mathbb{R}}

\newcommand{\T}{\mathsf{T}} 

\newcommand{\V}{\mathcal{V}}

\newcommand{\epi}{\operatorname{epi}}
\newcommand{\dom}{\operatorname{dom}}

\DeclareMathOperator{\bd}{bd}
\DeclareMathOperator{\ver}{vert}
\DeclareMathOperator{\Int}{int}

\DeclareMathOperator{\rec}{rec}
\DeclareMathOperator{\conv}{conv}

\DeclareMathOperator{\cone}{cone}

\newcommand{\norm}[1]{\left\lVert#1\right\rVert}
\newcommand{\abs}[1]{\lvert#1\rvert}

\newtheorem{theorem}{Theorem}
\numberwithin{equation}{section}
\numberwithin{theorem}{section}

\newtheorem{lemma}[theorem]{Lemma}
\newtheorem{corollary}[theorem]{Corollary}
\newtheorem{definition}[theorem]{Definition}
\newtheorem{remark}[theorem]{Remark}

\newtheorem{assumption}[theorem]{Assumption}

\title{{Algorithms for DC Programming via Polyhedral Approximations of Convex Functions}}
\author{Fahaar Mansoor Pirani \thanks{Bilkent University, Department of Industrial Engineering, Ankara, 06800, Turkey, fahaar.pirani@bilkent.edu.tr}
	\and Firdevs Ulus \thanks{Bilkent University, Department of Industrial Engineering, Ankara, 06800, Turkey, firdevs@bilkent.edu.tr} 
}
\date{\today}

\begin{document}
\maketitle

\begin{abstract} \noindent
	There is an existing exact algorithm that solves DC programming problems if one component of the DC function is polyhedral convex \cite{lohne2017solving}. Motivated by this, first, we consider two cutting-plane algorithms for generating an $\epsilon$-polyhedral underestimator of a convex function $g$. The algorithms start with a polyhedral underestimator of $g$ and the epigraph of the current underestimator is intersected with either a single halfspace (\Cref{alg_1}) or with possibly multiple halfspaces (\Cref{alg_1_mod}) in each iteration to obtain a better approximation. We prove the correctness and finiteness of both algorithms, establish the convergence rate of \Cref{alg_1}, and show that after obtaining an $\epsilon$-polyhedral underestimator of the first component of a DC function, the algorithm from \cite{lohne2017solving} can be applied to compute an $\epsilon$-solution of the DC programming problem without further computational effort. We then propose an algorithm (\Cref{alg_3}) for solving DC programming problems by iteratively generating a (not necessarily $\epsilon$-) polyhedral underestimator of $g$. We prove that \Cref{alg_3} stops after finitely many iterations and it returns an $\epsilon$-solution to the DC programming problem. Moreover, the sequence $\{x^k\}_{k \geq 0}$ outputted by \Cref{alg_3} converges to a global minimizer of the DC problem when $\epsilon$ is set to zero. Computational results based on some test instances from the literature are provided.
	
	\medskip
	
	\noindent
	{\bf Keywords:} DC Programming $\cdot$ Global optimization $\cdot$ Polyhedral approximation $\cdot$ Algorithms
	
	\medskip
	
	\noindent
	{\bf Mathematics Subject Classification:} 90C26 $\cdot$ 90C30 $\cdot$ 52B55
\end{abstract}

\section{Introduction} \label{sect:intro}

This paper is concerned with the difference of convex (DC) programming problems. A function $f: \mathbb{R}^n \rightarrow \mathbb{R}$ is a DC function if it can be {written} as $f = g - h$, where $g:\mathbb{R}^n \rightarrow \mathbb{R}, h: \mathbb{R}^n \rightarrow \mathbb{R}$ are convex. We consider the following DC programming problem 
\begin{equation*}
	\tag{P}
	\label{dc}
	\min\limits_{x \in X} (g(x)-h(x)),
\end{equation*}
where $X\subseteq \R^n$ is a convex compact set and $g$ and $h$ are convex functions on $X$.  

DC programming has been very useful in solving non-convex problems from many fields of applied sciences including data science, communication systems, biology, finance, logistics, and supply chain management. Many solution approaches have been developed to solve these problems over the last four decades. See, for instance, the review papers \cite{le2018dc, horst1999dc} for more details on DC programming and different solution approaches. 

There are also recent DC algorithms that are mainly extensions of the classical DC Algorithm (DCA) \cite{le2023open,aragon2018accelerating,aragon2020boosted,gotoh2018dc, lu2019nonmonotone}. Most of the existing algorithms guarantee approaching a local minimum. Indeed, the problem of computing a global minimum of DC programming problems is known to be NP-hard~\cite{ferrer2015solving}. Nevertheless, there exist algorithms for finding a global minimum of a DC programming problem including the DC extended cutting angle method (DCECAM) proposed in \cite{ferrer2015solving}. DCECAM is designed by adapting the extended cutting angle method of solving convex programming problems, and it works by iteratively generating a piecewise linear underestimate of the first component of the DC function. 

There are also exact algorithms for globally solving polyhedral DC programming problems, i.e., the problems where one of the component functions $g$ or $h$ is polyhedral convex. In 2017, Löhne and Wagner \cite{lohne2017solving} proposed exact solution methods that work by either solving an associated polyhedral projection problem (if $g$ is polyhedral convex) or by additionally solving finitely many convex programs (if $h$ is convex). In \cite{ciripoi2018vector}, the results of \cite{lohne2017solving} are improved further, and in \cite{vom2020solving} solution methods based on the concave minimization techniques from \cite{ciripoi2018vector} are proposed to solve polyhedral DC programming problems. 

Motivated by the results of \cite{lohne2017solving}, in this paper, we first consider two cutting-plane algorithms (Algorithms \ref{alg_1} and \ref{alg_1_mod}) to generate polyhedral convex underestimators to convex functions over a compact set such that the gap between the function and its underestimator is bounded by a predetermined tolerance. The idea is then to transform the DC programming problem into an approximate polyhedral DC programming problem and use the existing approaches to solve the approximate problem. Note that cutting-plane algorithms have existed in the literature for more than 60 years and have been used in solving many types of optimization problems, see for instance \cite{bertsekasBook, bertsekas2011unifying, goffin2002convex, kelley1960cutting}. In this study, we utilize a vertex enumeration solver, \emph{bensolve tools} \cite{ciripoi2018vector,lohne2015bensolve,lohne2016equivalence}, to implement two variants of cutting plane methods. The algorithms iteratively generate supporting hyperplanes to the epigraph of the convex function $g$. They start with an initial polyhedral underestimator of $g$ and in each iteration, they compute the vertices of the epigraph of the current underestimator. Using the vertex that is farthest away from the epigraph of $g$ (\Cref{alg_1}) or the set of all vertices that are farther away than a predetermined distance to the epigraph of $g$ (\Cref{alg_1_mod}), they update the underestimator until the approximation error is smaller than the predetermined level. 

We prove that \Cref{alg_1} is correct and we also estimate the convergence rate of it. We also prove the correctness and finiteness of \Cref{alg_1_mod}. Note that the approach for proving the convergence rate of \Cref{alg_1} cannot be directly applied to establish the convergence rate of \Cref{alg_1_mod}, hence this is left as a future work.

Algorithms \ref{alg_1} and \ref{alg_1_mod} are naive approaches for solving the DC program \eqref{P} as they are designed for generating a polyhedral \emph{approximation} of a convex function over the whole feasible set $X$ to then apply the method from \cite{lohne2017solving}. Next, we propose another algorithm, \Cref{alg_3}, to solve \eqref{P} in a more direct sense. \Cref{alg_3} also generates polyhedral underestimators to $g$, iteratively. However, it updates the current underestimator $\bar{g}$ while searching for an $\epsilon$-solution of the polyhedral DC programming problem $\min_{x \in X} (\bar{g}(x)-h(x))$. The resulting underestimator of $g$ is not necessarily an approximation of it for a given tolerance as in Algorithms \ref{alg_1} and \ref{alg_1_mod}. Note that similar approaches, in which different types of underestimators are generated using various choices of support functions of $g$, are proposed in the literature, see for instance \cite{beliakov2005review}.

We show that \Cref{alg_3} works correctly: for a predetermined tolerance $\epsilon > 0$, it terminates after finitely many iterations and returns a global $\epsilon$-solution of \eqref{P}. Moreover, any limit point of the sequence $\{x^k\}_{k \geq 0}$ outputted by \Cref{alg_3} is shown to be a global optimal solution of the DC programming problem \eqref{P}.
	
The rest of the paper is as follows. In \Cref{prelim}, we introduce notations, recall some basic concepts from convex analysis, and introduce the problem together with some well-known definitions and results. \Cref{sect:alg1} is devoted to Algorithms \ref{alg_1} and \ref{alg_1_mod}. This includes the correctness results and convergence analysis of these algorithms. In \Cref{sect:alg_2}, we explain \Cref{alg_3}, show its correctness and finiteness, and provide convergence results. We discuss the computational performance of the proposed algorithms on several examples in \Cref{sect:ex}, and future research directions in \Cref{conc}.

\section{Preliminaries and problem definition}
\label{prelim}

For $n\in \N$, let $\mathbb{R}^n$ denote the $n$-dimensional Euclidean space and $e^i \in \mathbb{R}^n$ be the unit vector given by $e_i^i = 1$ and $e^i_j = 0$ for all $j \neq i$.

Let $A, B \subseteq \mathbb{R}^n$ be nonempty sets and $\mu \in \mathbb{R}\setminus \{0\}$. Set operations are defined as $A+B \coloneqq\{x_1+x_2 \mid x_1 \in A, x_2 \in B\}$, $\mu A \coloneqq \{\mu x \mid x \in A\}$.  The convex hull, convex conic hull, interior, 
and boundary of $A$ are denoted by $\conv A, \cone A, \Int A,$ 
and $\bd A$, respectively. A recession direction of $A$ is a vector $k \in \mathbb{R}^n\setminus\{0\}$ satisfying $A + \cone \{k\} \subseteq A$. {The \textit{recession cone} of $A$ is the set of all recession directions of $A$,} $\rec A = \{k \in \mathbb{R}^n \mid \forall a \in A, \forall \mu \geq 0 : a + \mu k \in A\}$.

For a convex set $A$, $x \in A$, and $w \in \mathbb{R}^n\setminus \{0\}$, if $w^\T x = \inf_{z \in A} w^\T z$, then the set $\{z \in \mathbb{R}^n \mid w^\T z =  w^\T x \}$ is a \textit{supporting hyperplane} of $A$ at $x$. The set $\{z \in \mathbb{R}^n \mid w^\T z \geq  w^\T x \}\supseteq A$ is a \textit{supporting halfspace} of $A$ at $x$. A nonempty closed polyhedral convex set $A$ can be represented as the intersection of a finite number of halfspaces, that is, as $A = \bigcap_{i=1}^{r}\{y \in \mathbb{R}^n \mid (w^i)^\T y \geq a^i\}$ for some $r \in \mathbb{N}, w^i \in \mathbb{R}^n \setminus \{0\}$ and $a^i \in \mathbb{R}$ (\textit{H-representation} of $A$) or by its finitely many vertices $\{y^1,\dots,y^s\} \subseteq \mathbb{R}^n$ and directions $\{d^1,\dots,d^t\} \subseteq \mathbb{R}^n$ via $A = \conv \{y^1,\dots,y^s\} + \cone \{d^1,\dots,d^t\}$ (\textit{V-representation} of $A$). Throughout the paper, the set of vertices of $A$ is denoted by $\ver A \subseteq \R^n$.

Let $\bar{\mathbb{R}}$ denote the extended real line, that is $\bar{\mathbb{R}} \coloneqq \mathbb{R} \cup \{\pm \infty\}$ and $f : \mathbb{R}^n \rightarrow \bar{\mathbb{R}}$ be a function. The \emph{effective domain} of $f$ is $\dom f \coloneqq \{x \in \mathbb{R}^n \mid f(x) < \infty\}$. The function $f$ is said to be \emph{proper} if there exists some point $x_0 \in \dom f$ such that $f(x_0) \in \mathbb{R}$. The \emph{epigraph} of $f$ is $\epi f \coloneqq \{(x,r) \in \mathbb{R}^n \times \mathbb{R} \mid f(x) \leq r\} \subseteq \R^{n+1}$. The function $f$ is said to be \emph{closed} if $\epi f$ is closed, and \textit{polyhedral convex} if $\epi f$ is a polyhedral convex set. Let $x_0 \in \dom f$. The set $\partial f(x_0) \coloneqq \{c \in \mathbb{R}^n \mid \forall x \in \mathbb{R}^n \colon f(x) \geq f(x_0) + c^\T(x-x_0)\}$ is the \textit{subdifferential} of $f$ at $x_0$. An arbitrary element of $\partial f(x_0)$ is called a \emph{subgradient} of $f$ at $x_0$ and is denoted by $c(x_0)$, throughout. If $f$ is a proper closed convex function, then it is the pointwise supremum of all of its affine minorants, that is, $f(x)= \sup \{h(x)\mid h:\R^n \to \R \text{ is affine, }h \leq f\}$ for all $x \in \dom f$.

On $\mathbb{R}^n$, let $\norm{\cdot}$ be an arbitrary norm and $\norm{\cdot}_*$ be its dual norm. The conjugate function of the norm function can be written in terms of its dual norm as follows  \begin{equation*}
		\norm{\cdot}^*(y) \eqqcolon \sup_{x \in \R^n}\{y^\T x - \norm{x}\} = \begin{cases}
			0, & \text{if} \norm{y}_* \leq 1 \\
			+\infty, & \text{if} \norm{y}_* > 1.
		\end{cases}
\end{equation*}
The closed ball centered at $x \in \mathbb{R}^n$ having radius $\epsilon >0$ is $\mathbb{B}[x,\epsilon] \coloneqq \{z \in \mathbb{R}^n \mid \norm{z-x} \leq \epsilon\}$. For every $y \in \mathbb{R}^n$, the distance from a point $y$ to a set $A \subseteq \mathbb{R}^n$ is $d(y,A) := \inf_{x \in A} \norm{y-x}$. The Hausdorff distance between $A,B  \subseteq \mathbb{R}^n$ is defined as 
\begin{equation*}
	\delta^H(A,B) = \max\{\sup\limits_{x \in A}d(x,B), \sup\limits_{y \in B}d(y,A)\}.
\end{equation*}
The following lemma will be useful throughout the paper.
\begin{lemma}\cite[Lemma 2.2]{keskin2023outer}
	\label{hausdorff}
	Let $A$ and $B$ be convex sets in $\mathbb{R}^n$ with $\rec A = \rec B$ and $B \subseteq A$. If $A$ is polyhedral convex with at least one vertex, then $\delta^H(A,B) = \max_{v \in \ver A}d(v,B)$.
\end{lemma}

\subsection{Problem Definition}\label{Prob}

We are interested in obtaining the global minimum of a DC function, over a convex compact set $X \subseteq \mathbb{R}^n$, that is, solving the problem 
\begin{equation*}
	\tag{$P$}
	\label{P}
	\min_{x \in X} (g(x)-h(x)),
\end{equation*}
where $g:\R^n\to \R$ and $h:\R^n \to \R$ are convex functions. 
\begin{assumption}\label{assump}
We assume that $X\subseteq \R^n$ is a compact box with a nonempty interior, that is, $X=\{x\in \R^n \mid \forall i \in \{1,\ldots, n\}: \:\ell_i \leq x_i \leq u_i\}$ for some $\ell,u\in \R^n$ such that $\ell_i < u_i$ for all $i\in\{1,\ldots,n\}$. 
\end{assumption}

The existence of a solution to the problem \eqref{P} is known under \Cref{assump}. In this paper, the aim is to find a near-optimal solution in the sense of the following definition. 
	\begin{definition}\label{defn:eps}
	Let $\epsilon > 0$. $\bar{x} \in X$ is an $\epsilon$-solution to problem \eqref{P} if $$g(\bar{x})-h(\bar{x}) \leq \min_{x \in X} (g(x)-h(x)) + \epsilon.$$
\end{definition}

The following lemma and remark are simple observations and are included here since they will be useful for the design of the proposed algorithms.
\begin{lemma}
		\label{lem:hyp}
		Let $g:\mathbb{R}^n \rightarrow \mathbb{R}$ be convex and $x^* \in \mathbb{R}^n$.  Then, 
		\begin{equation*}
			H(g,x^\ast):=\{(x,t) \in \mathbb{R}^n \times \mathbb{R}\mid t- c(x^\ast)^\T x \geq g(x^*)-c(x^\ast)^\T x^*\}
		\end{equation*}
		is a supporting halfspace to $\epi g$ at $(x^*,g(x^*))$, where $c(x^\ast) \in \partial g(x^*)$ is a subgradient of $g$ at $x^*$.
\end{lemma}

\begin{remark}\label{rem:func_s}
Let $H(g,x^\ast)$ be as given in \Cref{lem:hyp} for a convex function $g$ and $x^*\in \R^n$. Let $s:\R^n \to \R$ be a linear function given by $s(x) \coloneqq g(x^\ast)+c(x^\ast)^\T(x-x^\ast)$. Then, $H(g,x^\ast) = \epi s$ and $s(x) \leq g(x)$ for all $x\in \R^n$.
\end{remark}

\section{Algorithms for approximating a convex function}\label{sect:alg1}
As mentioned in \Cref{sect:intro}, L\"ohne and Wagner~\cite{lohne2017solving} proposed an exact algorithm to solve the problem~\eqref{P} if at least one of $g$ or $h$ is a polyhedral convex function. The main idea of the first solution approach that we propose is to find a polyhedral approximation $\bar{g}$ of $g$ over $X$ and to use the algorithm from \cite{lohne2017solving} for finding an exact solution of the problem 
\begin{equation} \label{Pg}
	\min_{x\in X} (\bar{g}(x)-h(x)). \tag{$P_{\bar{g}}$} 
\end{equation}
Similarly, it is possible to obtain a polyhedral approximation $\bar{h}$ of $h$ and to solve 
\begin{equation} \label{Ph}
	\min_{x\in X} ({g}(x)-\bar{h}(x)). \tag{$P_{\bar{h}}$}
\end{equation}

To start with, we define a polyhedral approximation of a convex function with required properties as follows.

\begin{definition}
	\label{stop}
	Let $\epsilon >0$ and $g:\R^n \rightarrow \mathbb{R}$ be a convex function on a convex set $X \subseteq \mathbb{R}^n$. A polyhedral convex function $\bar{g}:\mathbb{R}^n \rightarrow \mathbb{R}$ is called an \emph{$\epsilon$-polyhedral underestimator of $g$ on $X$} if, for all $x \in X$, it satisfies
	\begin{equation*}
		0 \leq g(x)-\bar{g}(x) \leq \epsilon.
	\end{equation*}
\end{definition} 

Next, we show that if $\bar{g}$ (resp. $\bar{h}$) is an $\epsilon$-polyhedral underestimator of $g$ (resp. $h$) on $X$, then solving \eqref{Pg} (resp. \eqref{Ph}) yields an $\epsilon$-solution to problem \eqref{P}. 

\begin{theorem}\label{thm:alg1_opgap}
	Let $\epsilon > 0$, $g:\mathbb{R}^n \rightarrow \mathbb{R}$ and $h:\mathbb{R}^n \rightarrow \mathbb{R}$ be convex functions on a convex compact set $X \subseteq \mathbb{R}^n$. Let $\bar{g}:\mathbb{R}^n \rightarrow \mathbb{R}$ and $\bar{h}:\mathbb{R}^n \rightarrow \mathbb{R}$ be $\epsilon$-polyhedral underestimators of $g$ and $h$ over $X$. Let $x^g, x^h \in X$ be optimal solutions to problems \eqref{Pg}, \eqref{Ph}; and $z^*,z^g$ and $z^h$ be the optimal values of the problems \eqref{P}, \eqref{Pg} and \eqref{Ph}, respectively. Then, $x^g$ and $x^h$ are $\epsilon$-solutions of \eqref{P}. Moreover, $0 \leq z^*-z^g \leq \epsilon$ and $0 \leq z^h-z^* \leq \epsilon$ hold. 
\end{theorem}
\begin{proof}
	Note that $z^g \leq z^*$ holds since $\bar{g}(x) \leq g(x)$ holds for all $x \in X$. Then $x^g$ is an $\epsilon$-solution of \eqref{P} since we have $g(x^g)-h(x^g) \leq \bar{g}(x^g) - h(x^g) + \epsilon = z^g + \epsilon \leq z^\ast +\epsilon$. Note that the first inequality holds as $\bar{g}$ is an $\epsilon$-polyhedral underestimator of $g$. Moreover, $	z^*-z^g \leq \epsilon$ holds since	
	\begin{align}
		\label{g}
		z^*-z^g &= \inf_{x \in X}\{g(x)-h(x)\}-\inf_{x \in X}\{\bar{g}(x)-h(x)\}\nonumber\\
		&= -\sup_{x \in X}\{h(x)-g(x)\}+\sup_{x \in 	X}\{h(x)-\bar{g}(x)\}\nonumber\\
		&\leq \sup_{x \in X}\{g(x)-\bar{g}(x)\} \leq \epsilon. \nonumber
	\end{align}
	On the other hand, $z^* \leq z^h$ holds since  $\bar{h}(x) \leq h(x)$ holds for all $x \in X$. Moreover, similar to the previous case, we have 
	\begin{align}
		z^h-z^* &= \inf_{x \in X}\{g(x)-\bar{h}(x)\}-\inf_{x \in X}\{g(x)-h(x)\}\nonumber\\
		&= -\sup_{x \in X}\{\bar{h}(x)-g(x)\}+\sup_{x \in X}\{h(x)-g(x)\}\nonumber\\
		&\leq \sup_{x \in X}\{h(x)-\bar{h}(x)\}\leq \epsilon. \nonumber
	\end{align}
Finally, $x^h$ is an $\epsilon$-solution since $g(x^h)-h(x^h) \leq g(x^h)-\bar{h}(x^h) = z^h \leq z^\ast + \epsilon$ holds.
\end{proof}

\Cref{thm:alg1_opgap} suggests that after computing an $\epsilon$-polyhedral underestimator of $g$ or $h$, one can directly use the primal or dual methods from \cite{lohne2017solving} to solve the problems \eqref{Pg} or \eqref{Ph} and find $\epsilon$-solutions to problem \eqref{P}. 

\subsection{\Cref{alg_1}}\label{base}
Now, we describe the proposed algorithm which computes an $\epsilon$-polyhedral underestimator of a convex function $g$ over a box $X \subseteq \mathbb{R}^n$, see \Cref{assump}, for any precision level $\epsilon > 0$. The main idea is to approximate the epigraph of $g$ over $X$. To that end, we define the set to be approximated as 
\begin{equation}\label{C}
	\mathcal{C}\coloneqq \epi g \cap (X \times \R) \subseteq \R^{n+1}.
	\end{equation} 
Note that as $X \subseteq \R^n$ is compact, the recession cone of $\mathcal{C}$ is \begin{equation} \label{eq:K}
	\rec \mathcal{C} =   K\coloneqq \{(0,k) \in \R^n \times \R \mid k\geq 0\} = \cone \{e^{n+1}\}.
	\end{equation}

To initialize the algorithm, we start with some $x^0 \in \Int X$, and compute an outer approximation of $\mathcal{C}$ as 
	\begin{align} \label{eq:C0}
		C^0 \coloneqq H(g,x^0) \cap (X \times \R),
	\end{align} 
where $H(g,x^0)$ is as in \Cref{lem:hyp}. Clearly, $C^0$ is a convex polyhedral set and by \Cref{lem:hyp}, $C^0 \supseteq \mathcal{C}$. By construction, the recession cone of $C^0$ is also $K$. Then, we have $C^0 = \conv\ver C^0 + K$. Moreover, using \Cref{rem:func_s}, we also know that
\begin{equation} \label{eq:g0}
		g^0(x)\coloneqq g(x^0)+c(x^0)^\T(x-x^0),
\end{equation}
where $c(x^0) \in \partial g(x^0)$, is an underestimator of $g$ such that $\epi g^0 \cap (X \times \R) = C^0$.

At iteration $k$ the algorithm computes $\max_{x \in X}(g(x)-g^{k}(x))$. Since $g^{k}$ is a polyhedral convex function, by \cite[Corollary 10]{lohne2017solving}, an optimal solution exists among the vertices of its epigraph. By the construction of the algorithm, we have  $\epi g^k \cap (X \times \R) = C^k$ for every $k$. Hence an optimal solution $\bar{x}$ is computed as $(\bar{x},\bar{y}) \in \argmax_{(x,y) \in \ver C^{k}}(g(x)-y)$. The algorithm stops if $g(\bar{x})-g^{k}(\bar{x}) \leq \epsilon$, and returns $g^k$ and $\ver C^k$. Otherwise, a supporting halfspace $H(g,\bar{x})$ to $\epi g$ at $(\bar{x},g(\bar{x}))$ is generated. The current outer approximation of $\mathcal{C}$ and the polyhedral underestimator of $g$ are updated accordingly. See \Cref{alg_1} for the details. 

	\begin{algorithm}[H]
		\caption{Algorithm to compute an $\epsilon$-polyhedral underestimator of $g$.}
		\begin{algorithmic}[1] \label{alg_1}
			\STATE \textbf{Input:} $g:\R^n \to \R, X = [\ell, u] \subseteq \mathbb{R}^n, \epsilon > 0$.
			\STATE Set $k = 0$; 
			\STATE $ x^0 :=\frac{\ell+u}{2}$, set $C^0$ and $g^0$ as in \eqref{eq:C0} and \eqref{eq:g0}, respectively, and let $\mathcal{S} = \{g^0(x)\}$;		
				\WHILE{$\TRUE$}
				\STATE Compute $\ver C^k$; 
				\STATE Let $(\bar{x}, \bar{y}) \in \argmax_{(x,y) \in \ver C^{k}}(g(x)-y)$;
				\IF{$g(\bar{x}) - \bar{y} > \epsilon$}
				\STATE $C^{k+1} \coloneqq C^{k}\cap H(g,\bar{x})$; \\
				\STATE $\bar{s}(x) = g(\bar{x})+{c(\bar{x})}^\T(x-\bar{x})$, {$\mathcal{S} \gets \mathcal{S} \cup \{\bar{s}(x)\}$}; \\
				\STATE $g^{k+1}(x) \gets \max\{g^{k}(x),\bar{s}(x)\}$;
				\STATE $k\gets k+1;$
				\ELSE
				\STATE break;
				\ENDIF
				\ENDWHILE
				\RETURN $\begin{cases} 
					g^k: \text{an $\epsilon$-polyhedral underestimator of $g$.}  \\
					{\ver C^k}: \text{vertices of } \epi g^k \cap (X \times \R).
				\end{cases}$
			\end{algorithmic}
	\end{algorithm}

The next theorem states that when \Cref{alg_1} terminates, it returns an $\epsilon$-polyhedral underestimator of $g$. 
	\begin{theorem}\label{thm:alg1correct}
		Let $g: \R^n \to \R$ be a convex function and $\epsilon > 0$. When \Cref{alg_1} stops, it returns an $\epsilon$-polyhedral underestimator of $g$ on $X$.
	\end{theorem}
	\begin{proof}
		By \Cref{lem:hyp}, $C^0 \supseteq \mathcal{C}$. Moreover, by construction and \cite[Corollary 18.5.3]{rockafellar1997convex}, $\ver C^k \ne \emptyset$ for all $k\geq 0$. By \Cref{lem:hyp}, $H(g,\bar{x}) \supseteq \mathcal{C}$, hence, $C^k \supseteq \mathcal{C}$ for all $k \geq 0$ through the algorithm. By \Cref{rem:func_s} and by construction of the algorithm, $g^k$ is a polyhedral underestimator  of $g$ for any $k\geq 0$. Moreover, we have ${C^k} = \epi g^k \cap (X \times \R)$, in particular, for every $(\bar{x},\bar{y}) \in \ver {C^k}$, we have $\bar{y} = g^k (\bar{x})$. On the other hand, for every $\bar{x} \in X$, $(\bar{x},g^k(\bar{x})) \in \bd {C^k}$.
		
		Assume that \Cref{alg_1} stops and returns $g^{\bar{k}}$ for some $\bar{k} \geq 1$. Then, $g^{\bar{k}}$ is an $\epsilon$-polyhedral underestimator of $g$ on $X$, since we have
		\begin{align}
			\max\limits_{x \in X}(g(x)-g^{\bar{k}}(x))  = \max_{(x,y) \in \ver C^{\bar{k}}}(g(x)-y) = g(\bar{x})-g^{\bar{k}}(\bar{x}) \leq \epsilon\nonumber,
		\end{align} 
		where the first equality is by \cite[Corollary 10]{lohne2017solving}. The second equality and the last inequality follow from line 6 and lines 7, 13 of \Cref{alg_1}, respectively.
\end{proof}

Next, we study the convergence of \Cref{alg_1}. For the main results of this section, we assume that \Cref{alg_1} is run for a closed proper convex function $g:\R^n\to \R$. Moreover, we assume that $g$ is non-polyhedral and the algorithm is run with $\epsilon = 0$. This ensures that the algorithm runs indefinitely while updating the current underestimator at each iteration. 

To establish the convergence rate of \Cref{alg_1}, we use the convergence results of a method for approximating convex compact sets from \cite{lotov2004interactive}. For a compact convex set $\mathcal{A}\subseteq \R^{n+1}$, a sequence of outer approximating polytopes $\mathcal{A}_k$, $k\geq 0$ satisfying $\mathcal{A}_0\supseteq \mathcal{A}_1 \supseteq \ldots \mathcal{A}$ is said to be generated by a \emph{cutting method} if 
\begin{enumerate}[1.]
	\item $\mathcal{A}_0 \supseteq \mathcal{A}$ is a polyhedral set which is an intersection of supporting halfspaces of $\mathcal{A}$; and
	\item $\mathcal{A}_{k+1} = \mathcal{A}_k \cap H_k$ for all $k\geq 0$, where $H_k$ is a supporting halfspace of $\mathcal{A}$. 
\end{enumerate} 

The following definition and theorem from \cite{lotov2004interactive} will be used to estimate the convergence rate of \Cref{alg_1}.
	\begin{definition}\cite[Definition 8.3]{lotov2004interactive}
		Let $\mathcal{A}\subseteq \R^{n+1}$ be a compact convex set and $\mathcal{A}_k$, $k\geq 0$ be generated by a cutting method. $(\mathcal{A}_k)_{k\geq 0}$ is called an $H(\gamma,\mathcal{A})$-sequence of cutting if there exists a constant $\gamma > 0$ such that for any $k \geq 0$ it holds that
		\begin{equation*}
			\delta^H(\mathcal{A}_k,\mathcal{A}_{k+1}) \geq \gamma\delta^H(\mathcal{A}_k,\mathcal{A}).
		\end{equation*}
	\end{definition}
	\begin{theorem}\cite[Theorems 8.5, 8.6]{lotov2004interactive}
		\label{thm:conv_rate_lotov}
		Let $\gamma >0$, $\mathcal{A}\subseteq \R^{n+1}$ be a convex compact set and $(\mathcal{A}_k)_{k \geq 0}$ be an $H(\gamma,\mathcal{A})$-sequence of cutting. Then for any $0 < \epsilon < 1,$ there exists $N \in \mathbb{N}$ such that for $k \geq N$ it holds that
		\begin{equation*}
			\delta^H(\mathcal{A}_k,\mathcal{A}) \leq (1+\epsilon)\lambda(\gamma,\mathcal{{A}})k^{-\frac{1}{n}},
		\end{equation*}
		where $\lambda(\gamma,\mathcal{{A}}) >0$ is a parameter that depends on the topological properties of $\mathcal{A}$ together with $\gamma$. In particular, $\lim_{k \rightarrow \infty}\delta^H({\mathcal{A}_k,\mathcal{A}}) = 0$ holds.
	\end{theorem}

For the convergence rate of \Cref{alg_1}, we work with a convex compact subset $A$ of $\R^{n+1}$ which satisfies $A+K = \mathcal{C}$, where $K$ is the upward cone given as in \eqref{eq:K}. To this end, consider the halfspace given by 
\begin{equation}\label{eq:S}
	S:=\{(x^\T,t)^\T \in \mathbb{R}^{n}\times \R \mid t \leq b\},
\end{equation}
where $b \coloneqq \sup_{x \in X}g(x)\in \R$. It is not difficult to show that the set 
\begin{equation}\label{eq:A}
	{A:= \epi g \cap (X \times \mathbb{R}) \cap S = \mathcal{C} \cap S}
\end{equation} 
is a convex compact set satisfying $A+K = \mathcal{C}$. 

We also define the following sets
	\begin{equation}\label{eq:Ak1}
		{A^k:= \epi g^k \cap (X \times \mathbb{R}) \cap S = {C}^k \cap S},
	\end{equation}
	where $g^k$ for $k\geq 0$ are as in \Cref{alg_1}. 
	Similar to $A$, these are convex compact sets satisfying $A^{(\cdot)} + K = \epi g^{(\cdot)} \cap (X \times \mathbb{R})$. Moreover, $A \subseteq A^{k+1} \subseteq A^k$ holds for all $k\geq 0$.
	\begin{remark}\label{rem:vertAk1}
		A simple but important observation regarding the sets $A^{k}$ is that $\ver A^{k} = \ver C^k \cup (\ver X \times \{b\})$ for all $k \geq 0$. Moreover, $\ver X \times \{b\}\subseteq A$. This implies that for any $k\geq 0$, we have $$\max_{(x,y) \in \ver C^{k}}(g(x)-y)= \max_{(x,y) \in \ver A^{k}}(g(x)-y).$$ If the maximum is positive, then the arguments of the maxima are equal as well.
	\end{remark}

\begin{lemma}\label{lem:haus_A_Ak}
	Let $k \geq 0$, $(\bar{x}, \bar{y}) \in \argmax_{(x,y) \in \ver C^{k}}(g(x)-y)$ and $A, A^k$ be as given in \eqref{eq:A},\eqref{eq:Ak1}, respectively. Then, $\delta^H(A^k,A) \leq g(\bar{x})-\bar{y}$.
\end{lemma}
\begin{proof}
	The statement holds trivially if $g(\bar{x})-\bar{y}=0$ since it implies that $A^k=A$. On the other hand, from \Cref{hausdorff} and \Cref{rem:vertAk1}, we obtain 
	\begin{align}
		\delta^H(A^k,A) 
		&= \max_{(x,y) \in \ver A^k} \inf_{(x^a,y^a)\in A} \norm{(x^a,y^a)-(x,y)} \notag \\
		& \leq \max_{(x,y) \in \ver A^k} \norm{(x,g(x))-(x,y)} \notag \\
		& = \max_{(x,y) \in \ver A^k} (g(x)- y)\notag \\
		& = \max_{(x,y) \in \ver C^k} (g(x)-y) = g(\bar{x})-\bar{y}.
	\end{align}
\end{proof} 
\begin{theorem}	\label{thm:Hgamma}
	Assume $g:\R^n \to \R$ is a closed proper convex function. Let $A, (A^k)_{k\geq 0}$ be as given in \eqref{eq:A}, and \eqref{eq:Ak1}, respectively. $(A^k)_{k\geq 0}$ is an $H(\gamma,A)$-sequence for some $\gamma > 0$.
\end{theorem}
\begin{proof}
	Let $k\geq 0$ be arbitrary and  $(\bar{x}, \bar{y}) \in \argmax_{(x,y) \in \ver A^{k}}(g(x)-y)$. By \Cref{rem:vertAk1}, $(\bar{x}, \bar{y}) \in \argmax_{(x,y) \in \ver C^{k}}(g(x)-y)$ and $\bar{y} = g^k (\bar{x})$. Indeed, \Cref{alg_1} considers $(\bar{x},\bar{y})$ at $k^{th}$ iteration and $A^{k+1} = A^k \cap H(g,\bar{x})$, where $H(g,\bar{x})=\{(x,t) \in \mathbb{R}^n \times \mathbb{R}\mid t- c(\bar{x})^\T x \geq g(\bar{x})-c(\bar{x})^\T \bar{x}\}$
	is a supporting halfspace to $\epi g$ at $(\bar{x},g(\bar{x})) \in A$. Here, $c(\bar{x}) \in \partial g(\bar{x})$ is a subgradient of $g$ at $\bar{x}$, see \Cref{lem:hyp}. Let $(x^\prime,y^\prime)\in A^{k+1}$ be arbitrary and $m:= \frac{(-c(\bar{x})^\T,1)^\T}{\norm{(-c(\bar{x})^\T,1)}_\ast}$. Then, $m^\T(x^\prime,y^\prime) \geq m^\T(\bar{x},g(\bar{x}))$ implies 
	\begin{align}
		m^\T((x^\prime,y^\prime)-(\bar{x},\bar{y})) &\geq 
		m^\T (0,g(\bar{x})-\bar{y})
		= \frac{g(\bar{x})-\bar{y}}{\norm{(-c(\bar{x})^\T,1)}_\ast} \geq \frac{\delta^H(A^k,A)}{\norm{(-c(\bar{x})^\T,1)}_\ast}, \notag 
	\end{align} 
	where the last inequality is by \Cref{lem:haus_A_Ak}. On the other hand, from H\"{o}lder's inequality, we have \begin{equation*}
		m^\T((x^\prime,y^\prime)-(\bar{x},\bar{y})) \leq \norm{m}_\ast \norm{(x^\prime,y^\prime)-(\bar{x},\bar{y})} = \norm{(x^\prime,y^\prime)-(\bar{x},\bar{y})}.
	\end{equation*} Then, $d((\bar{x},\bar{y}),A^{k+1}) = \inf_{(x^\prime,y^\prime)\in A^{k+1}} \norm{(x^\prime,y^\prime)-(\bar{x},\bar{y})}\geq \frac{\delta^H(A^k,A)}{\norm{(-c(\bar{x})^\T,1)}_\ast}.$ From \Cref{hausdorff}, we obtain $$\delta^H(A^k,A^{k+1}) = \max_{v \in \ver A^k}d(v,A^{k+1}) \geq d((\bar{x},\bar{y}),A^{k+1}) \geq \frac{\delta^H(A^k,A)}{\norm{(-c(\bar{x})^\T,1)}_\ast} .$$
	From \cite[Theorem 24.7]{rockafellar1997convex}, $\bigcup_{x\in X} \partial g(x)$ is a nonempty compact set. This implies for some $\gamma > 0$ that $\sup_{x\in X, c(x)\in \partial g(x)} \norm{(-c(x)^\T,1)}_\ast \leq \frac{1}{\gamma}$ holds. 
\end{proof}
\begin{corollary}
	\label{cor:convergence}
	Assume $g:\R^n \to \R$ is a closed proper convex function. Let $A, (A^k)_{k\geq 0}$ be as given in \eqref{eq:A}, and \eqref{eq:Ak1}, respectively.
	\begin{enumerate}[(a)]
		\item The approximation error for the sequence $(A^k)_{k\geq 0}$ decreases by the order $\mathcal{O}( k^{-\frac{1}{n}})$.
		\item $\lim_{k \rightarrow \infty}\delta^H(A^k,A) = 0$.
	\end{enumerate} 
\end{corollary}
\begin{proof}
	(a) By \Cref{thm:Hgamma}, $(A^k)_{k\geq 0}$ is an $H(\gamma,A)$-sequence of cutting for some $\gamma > 0$. Then by \Cref{thm:conv_rate_lotov}, for any $0 < \epsilon < 1,$ there exists $N \in \mathbb{N}$ such that for $k \geq N$ it holds that 
	\begin{equation*}
		\delta^H({A}^k,{A}) \leq (1+\epsilon)\lambda(\gamma,A)k^{-\frac{1}{n}}, \text{ where $\lambda(\gamma,A)$ is as given in \Cref{thm:conv_rate_lotov}}.
	\end{equation*} (b) follows directly from (a).
\end{proof}

\subsection{\Cref{alg_1_mod}}\label{mod}
In this section, we describe a modified version of \Cref{alg_1}. The motivation is to possibly reduce the computational time. To compute the vertices of a polyhedral set given by its H representation, we solve vertex enumeration problems, which are computationally expensive in general. In each iteration of \Cref{alg_1}, only a single halfspace is intersected with the epigraph of the current underestimator, and the vertex enumeration is applied for the updated set. Instead, at iteration $k$, \Cref{alg_1_mod} considers the set of all vertices of the current outer approximation {$C^k$}. If a vertex $(\bar{x},\bar{y})\in \R^n\times \R$ of {$C^k$} is sufficiently close to $\epi g$, it is added to set $\V$, which stores the set of sufficiently close vertices. Otherwise, a supporting halfspace to $\epi g$ at $(\bar{x},g(\bar{x}))$ is generated and stored. The current outer approximation of $\mathcal{C}$ is updated by intersecting it with all these supporting halfspaces at once. The polyhedral underestimator of $g$ is updated, accordingly. The algorithm terminates when all the vertices of $C^k$ are close to $\epi g$, see \Cref{alg_1_mod}.

\begin{algorithm}[!]
	\caption{Algorithm to compute an $\epsilon$-polyhedral underestimator of a given function.}
	\begin{algorithmic}[1] \label{alg_1_mod}
		\STATE \textbf{Input:} $g:\R^n \to \R, X = [\ell, u] \subseteq \mathbb{R}^n, \epsilon > 0$.
		\STATE Set $\mathcal{V} = \emptyset, k = 0, R = \emptyset$; 
		\STATE $ x^0 :=\frac{\ell+u}{2}$, set $C^0$ {and $g^0$} as in \eqref{eq:C0} and \eqref{eq:g0}, respectively and let $\mathcal{S} = \{g^0(x)\}$;		
			\WHILE{$R \neq \R^{n+1}$}
			\STATE Compute $\ver C^k$;
			\STATE $R = \R^{n+1}$, $j=0, \ g^{k,j}(x) \coloneqq g^k(x)$;
			\FORALL{$(\bar{x}, \bar{y}) \in \ver {C^{k}} \setminus \mathcal{V}$}
			\IF{$g(\bar{x}) - \bar{y} > \epsilon$}
			\STATE $R \gets R \cap H(g,\bar{x})$; \\
			\STATE $\bar{s}(x) = g(\bar{x})+{c(\bar{x})}^\T(x-\bar{x})$, {$\mathcal{S} \gets \mathcal{S} \cup \{\bar{s}(x)\}$}; \\
			\STATE $g^{k,j+1}(x) \gets \max\{g^{k,j}(x),\bar{s}(x)\}$, $j\gets j+1;$
			\ELSE
			\STATE $\mathcal{V} \leftarrow \mathcal{V} \cup \{\bar{x}\}$;
			\ENDIF
			\ENDFOR	
			\STATE {$C^{k+1} \coloneqq C^{k}\cap R$}, $g^{{k+1}}(x) \gets g^{k,j}(x)$, $J^k\coloneqq j$, {$k\gets k+1$};
			\ENDWHILE
			\RETURN $\begin{cases} 
				g^k: \text{an $\epsilon$-polyhedral underestimator of $g$.}  \\
				{\ver C^k}: \text{vertices of } \epi g^k \cap (X \times \R). 
			\end{cases}$ 	
		\end{algorithmic}
	\end{algorithm}

	The next theorem states that when \Cref{alg_1_mod} terminates, it returns an $\epsilon$-polyhedral underestimator of $g$. The proof is omitted as it is similar to the proof of \Cref{thm:alg1correct}.
	
	\begin{theorem}\label{thm:alg2correct}
			Let $g: \R^n \to \R$ be a convex function and $\epsilon > 0$. When \Cref{alg_1_mod} stops, it returns an $\epsilon$-polyhedral underestimator of $g$ on $X$.
	\end{theorem}
	
Next, we prove that \Cref{alg_1_mod} stops after finitely many iterations for any $\epsilon > 0$ if  $g$ is a closed proper convex function. Let $\mathcal{C},K, S, A$ and $(A^k)_{k \geq 0}$ be as in \eqref{C},  \eqref{eq:K}, \eqref{eq:S}-\eqref{eq:Ak1}. Recall that $A, A^k$ are convex compact sets in $\mathbb{R}^{n+1}$ satisfying $A^{(\cdot)}+K = \epi g^{(\cdot)} \cap (X \times \mathbb{R})$. Moreover, $A \subseteq A^{k+1} \subseteq A^k$ holds for all $k \geq 0$. 

Below, we provide two technical results, upon which the finiteness result is based. The following remark is an observation used to state the subsequent lemma.
	\begin{remark}\label{rem:gamma}
		If $g:\R^n \to \R$ is a closed proper convex function, then by \cite[Theorem 24.7]{rockafellar1997convex}, $\bigcup_{x \in X}\partial g(x)$ is a nonempty bounded closed subset. Hence, $\beta \coloneqq \sup_{x\in X, c(x)\in \partial g(x)} \norm{(-c(x)^\T,1)}_\ast$ is well defined and $\beta > 0$.  
	\end{remark}

\begin{lemma}\label{ball1}
	Assume $g:\R^n \to \R$ is a closed proper convex function. Fix $\epsilon > 0$. Let $\bar{v} = (\bar{x},\bar{g}(\bar{x})) \notin A$, where $\bar{g}$ is a polyhedral underestimator of $g$, and $H^{\epsilon}(g,\bar{x})$ be a halfspace defined by 
	\begin{equation}\label{H_eps}
		H^{\epsilon}(g,\bar{x}) \coloneqq \{s \in \mathbb{R}^{n+1} \mid m^\T s \geq m^\T \bar{a}-\frac{\epsilon}{2\beta}\},
	\end{equation} 
	where $\beta$ is as defined in \Cref{rem:gamma}, $m \coloneqq \frac{(-c(\bar{x})^\T,1)^\T}{\norm{(-c(\bar{x})^\T,1)}_*}$ and $\bar{a} = (\bar{x}, g(\bar{x}))$. If $g(\bar{x})-\bar{g}(\bar{x}) > \epsilon$, then $\mathbb{B}[\bar{v},\frac{\epsilon}{4\beta}] \cap H^{\epsilon}(g,\bar{x}) = \emptyset$.
\end{lemma}
\begin{proof}
	Let $s^* \in H^{\epsilon}(g,\bar{x})$ be arbitrary. We have $$m^\T s^* \geq m^\T (\bar{v}+(\bar{a}-\bar{v}))-\frac{\epsilon}{2\beta} = m^\T\bar{v} + \frac{g(\bar{x})-\bar{g}(\bar{x})}{\norm{(-c(\bar{x})^\T,1)}_*}-\frac{\epsilon}{2\beta} \geq  m^\T\bar{v}+ \frac{g(\bar{x})-\bar{g}(\bar{x})}{\beta}-\frac{\epsilon}{2\beta}.$$ Equivalently, $m^\T(s^*-\bar{v}) \geq \frac{g(\bar{x})-\bar{g}(\bar{x})}{\beta}-\frac{\epsilon}{2\beta}.$ Using $\norm{m}_* = 1$ and $g(\bar{x})-\bar{g}(\bar{x}) > \epsilon$, we obtain $s^* \notin \mathbb{B}[\bar{v},\frac{\epsilon}{4\beta}]$ as 
	\begin{equation*}
		\norm{s^*-\bar{v}} \geq |m^\T(s^*-\bar{v})| \geq \frac{g(\bar{x})-\bar{g}(\bar{x})}{\beta}-\frac{\epsilon}{2\beta} > \frac{\epsilon}{\beta}-\frac{\epsilon}{2\beta} = \frac{\epsilon}{2\beta}.
	\end{equation*}
\end{proof}
The following lemma can be found in \cite[Lemma 2.1]{ararat2023convergence}. It is restated in terms of the terminology used here.
\begin{lemma}\label{eq1:lem21}	
	Let $\bar{v} = (\bar{x},\bar{g}(\bar{x})) \notin A$ and $H(g,\bar{x})$ and $H^{\epsilon}(g,\bar{x})$ be halfspaces defined by \Cref{lem:hyp} and \eqref{H_eps} respectively. Then $H(g,\bar{x}) + \mathbb{B}[0,\frac{\epsilon}{2\beta}] \subseteq H^{\epsilon}(g,\bar{x})$.
\end{lemma}

\begin{theorem}\label{thm:finitealg2}
Assume $g:\R^n \to \R$ is a closed proper convex function.	For any $\epsilon > 0$, \Cref{alg_1_mod} terminates after a finite number of iterations.
\end{theorem}
\begin{proof}
	By the construction of the sets in \eqref{eq:Ak1}, the number of vertices of $A^k$ is finite for every $k \geq 0$. It is sufficient to prove that there exists a $k_{\epsilon} \geq 0$ such that for every vertex $v = (x, g^{k_{\epsilon}}(x)) \in \ver A^{k_{\epsilon}}$, we have $g(x)-g^{k_{\epsilon}}(x) \leq \epsilon$. Assume to the contrary that for every $k \geq 0$, there exists a vertex $v^k \in \ver A^k$ such that $g(x^k)-g^k(x^k) > \epsilon$. For the rest of the proof, we fix an arbitrary $v^k \in \ver A^k$ satisfying this condition. 

	Consider the compact set $A^0 + \mathbb{B}[0,\frac{\epsilon}{2\beta}]$. Define for an arbitrary $k \geq 0$, $B^k \coloneqq \{v^k\} +  \mathbb{B}[0,\frac{\epsilon}{4\beta}]$. Since $v^k \in A^k$, it holds true that
	\begin{equation}\label{BkinA0+Ball1}
		B^k \subseteq \{v^k\} + \mathbb{B}[0,\frac{\epsilon}{2\beta}] \subseteq A^k + \mathbb{B}[0,\frac{\epsilon}{2\beta}] \subseteq A^0 + \mathbb{B}[0,\frac{\epsilon}{2\beta}].
	\end{equation}	
	To prove $B^i \cap B^j = \emptyset$, for every $i,j \geq 0$ with $i \ne j$, without loss of generality, assume that $i < j$. Note that $A^{j} \subseteq A^{i+1}$. From \Cref{ball1}, we have $B^i \cap H^{\epsilon}(g,x^i) = \emptyset$. Moreover,
	\begin{equation*}\label{Aj+Binclusion}
		A^j + \mathbb{B}[0,\frac{\epsilon}{2\beta}] \subseteq A^{i+1} + \mathbb{B}[0,\frac{\epsilon}{2\beta}] = (A^i \cap \bigcap_{\substack{v \in \ver A^i \\ g(x^v)-g^i(x^v) > \epsilon}}H(g,x^v)) + \mathbb{B}[0,\frac{\epsilon}{2\beta}] \subseteq H(g,x^i) + \mathbb{B}[0,\frac{\epsilon}{2\beta}],
	\end{equation*}
	where $H(g,x^{i})$ is a supporting halfspace to $\epi g$ at $(x^i,g(x^i))$. Using \Cref{eq1:lem21}, we get
	\begin{equation*}
		A^{j} + \mathbb{B}[0,\frac{\epsilon}{2\beta}] \subseteq H(g,x^{i}) + \mathbb{B}[0,\frac{\epsilon}{2\beta}] \subseteq H^{\epsilon}(g,x^{i}).
	\end{equation*}
	This implies that $B^i \cap (A^{j}+\mathbb{B}[0,\frac{\epsilon}{2\beta}]) = \emptyset$. On the other hand, $B^{j} \subseteq A^{j} + \mathbb{B}[0,\frac{\epsilon}{2\beta}]$ from \eqref{BkinA0+Ball1}. Hence $B^i \cap B^{j} = \emptyset$. This is a contradiction as these imply that there is an infinite number of disjoint sets, with the same positive volume, contained in the compact set $A^0 + \mathbb{B}[0,\frac{\epsilon}{2\beta}]$.	
\end{proof}

\begin{remark}\label{rem:unifconv}
	If $ \epsilon = 0$ in \Cref{alg_1} (resp. \Cref{alg_1_mod})), then the sequence $\{g^k\}_{k \geq 0}$ outputted by the algorithm converges uniformly to $g$. The pointwise convergence follows from \Cref{cor:convergence} (resp. \Cref{thm:finitealg2}) and the uniform convergence holds as $X$ is compact. 
\end{remark}

\section{An Algorithm for solving DC programming problems}\label{sect:alg_2}
The solution methodology proposed in \Cref{sect:alg1} is a naive approach for solving the DC programming problems. To use the existing exact solution algorithm from \cite{lohne2017solving} for polyhedral DC programming problems, Algorithms \ref{alg_1} and \ref{alg_1_mod} return an $\epsilon$-polyhedral underestimator of a given convex function over a convex compact set $X$. In this section, we propose an algorithm (\Cref{alg_3})) to solve the general DC programming problems in a more direct sense. Even though the general idea is, in a way, similar to \Cref{alg_1}, \Cref{alg_3} does not compute an $\epsilon$-polyhedral underestimator of the convex function $g$ over the whole feasible set $X$. Instead, it keeps updating the underestimator locally while looking for an $\epsilon$-solution of the DC problem.

The next theorem will help explain the working mechanism of \Cref{alg_3}.
\begin{theorem}
	\label{thm:alg2}
	Let $\epsilon > 0$, $g : \mathbb{R}^n \rightarrow \mathbb{R}$, and $h : \mathbb{R}^n \rightarrow \mathbb{R}$ be convex functions on a convex compact set $X \subseteq \mathbb{R}^n$. Let $\bar{g}:\mathbb{R}^n \rightarrow \mathbb{R}$ be a polyhedral underestimator of $g$ over $X$. Let $\bar{x} \in X$ be an optimal solution to the problem \eqref{Pg}, and $z^*$, $\bar{z}$ be the optimal values of the problems \eqref{P},\eqref{Pg} respectively. 
	If $g(\bar{x})-\bar{g}(\bar{x}) \leq \epsilon$, then $\bar{x}$ is an $\epsilon$-solution of \eqref{P}. Moreover, $0 \leq z^*-\bar{z} \leq \epsilon$ holds.
\end{theorem}
\begin{proof}
	Note that $\bar{z} \leq z^*$ holds since $\bar{g}(x) \leq g(x)$ holds for all $x \in X$. Then, $\bar{x}$ is an $\epsilon$-solution of \eqref{P} since $g(\bar{x})-h(\bar{x}) \leq \bar{g}(\bar{x})-h(\bar{x}) + \epsilon = \bar{z} + \epsilon \leq  z^* + \epsilon.$ This also implies that $g(\bar{x})-h(\bar{x})-\bar{z} \leq \epsilon$. Then, as $z^*$ is the optimal value of \eqref{P}, we obtain, $ z^*- \bar{z} \leq g(\bar{x})-h(\bar{x})-\bar{z} \leq \epsilon$.   
\end{proof} 
As \Cref{thm:alg2} suggests, in \Cref{alg_3}, the aim is to generate a polyhedral underestimator $\bar{g}$ of the function $g$, such that $g(\bar{x}) - \bar{g}(\bar{x}) \leq \epsilon$, where $\bar{x}$ solves \eqref{Pg} optimally. To that end, we use some terminology as exactly they are used in \Cref{sect:alg1}. In particular, let $\mathcal{C}$ be as given in \eqref{C}. The initialization of \Cref{alg_3} is also the same as in \Cref{alg_1}. In particular, we set $\bar{C}^0 \coloneqq C^0$, see \eqref{eq:C0}, as the initial outer approximation of $\mathcal{C}$. Recall that $K$ is the recession cone of $\mathcal{C}$ and $\bar{C}^0$, see \eqref{eq:K}. Moreover, $\bar{C}^0 = \epi g^0 \cap (X \times \mathbb{R})$, where $g^0(x)$ is as in \eqref{eq:g0}. As will be explained below, the algorithm iterates by updating the epigraph of the current underestimator so that for each iteration $k$,  $\bar{C}^k = \epi g^k \cap (X \times \mathbb{R})$ holds true.

At iteration $k$, where $k\geq 1$, the algorithm considers the current underestimator $g^{k-1}$ of the function $g$ and computes an optimal solution to the following problem:
\begin{equation*}
	\tag{$P_k$}
	\label{eq:pk}
	\min_{x \in X}(g^{k-1}(x)-h(x)).
\end{equation*}
Since $g^{k-1}$ is a polyhedral convex function, the existence of an optimal solution among the vertices of $\bar{C}^{k-1}$ is guaranteed by \cite[Corollary 10]{lohne2017solving}. Hence, an optimal solution $x^k$ can be computed as $$(x^k,y^k) \in \argmin_{(x,y)\in \ver \bar{C}^{k-1}}(y-h(x)).$$ Note that $y^k = g^{k-1}(x^k)$ holds by construction. The algorithm checks if $g(x^k)-g^{k-1}(x^k) \leq \epsilon$. If this is the case, $x^k$ is returned. Otherwise, a supporting halfspace to $\epi g$ at $(x^k,g(x^k))$ is generated. The current outer approximation of $\mathcal{C}$ and the polyhedral underestimator of $g$ are updated accordingly, see \Cref{alg_3} for the details.

\begin{algorithm}[H]
	\caption{An algorithm to compute an $\epsilon$-solution to \eqref{P}}
	\begin{algorithmic}[1] \label{alg_3}
		\STATE \textbf{Input:} Problem \eqref{P}, $\epsilon > 0$.
		\STATE Set $k = 0$, $ x^0 :=\frac{l+u}{2}$, set $\bar{C}^0, g^0$ as in \eqref{eq:C0},\eqref{eq:g0}, respectively.
		\WHILE{$\TRUE$}
		\STATE $k \gets k+1$;
		\STATE Compute $\ver \bar{C}^{k-1}$, let $(x^k,y^k) \in \arg\min_{(x,y)\in \ver \bar{C}^{k-1}}(y-h(x))$ (i.e., solve \eqref{eq:pk});
		\IF{$g(x^k) - y^k > \epsilon$}
		\STATE $s^k(x) = g(x^k)+c(x^k)^\T(x-x^k)$; 
		\STATE $g^k(x) \gets \max\{g^{k-1}(x),s^k(x)\}$;
		\STATE $\bar{C}^k \coloneqq \bar{C}^{k-1}\cap H(g,x^k)$; 
		\ELSE
		\STATE break;
		\ENDIF
		\ENDWHILE
		\RETURN $ x^k: \epsilon\text{-solution of \eqref{P} }$.
	\end{algorithmic}
\end{algorithm}

First, we show that when terminates, \Cref{alg_3} returns an $\epsilon$-solution of \eqref{P}.

\begin{theorem}\label{alg3_correct}
	Let $\epsilon > 0$. When \Cref{alg_3} stops, it returns an $\epsilon$-solution of \eqref{P}.	
\end{theorem}
\begin{proof}
	For any $k\geq 0$, $\ver \bar{C}^k \ne \emptyset$, $g^k$ is a polyhedral underestimator  of $g$, and ${\bar{C}^k} = \epi g^k \cap (X \times \R)$ holds. Then, in line 5 of \Cref{alg_3}, the algorithm returns a solution $x^k$ to \eqref{eq:pk} by \cite[Corollary 11]{lohne2017solving}, where $y^k = g^{k-1}(x^k)$. If the algorithm stops at iteration $\bar{k}$ for some $\bar{k}\geq 1$, then $g(x^{\bar{k}}) - g^{\bar{k}-1}(x^{\bar{k}}) \leq \epsilon$ holds and by \Cref{thm:alg2}, $x^{\bar{k}}$ is an $\epsilon$-solution to \eqref{P}. 
\end{proof} 

Next, we study the convergence of \Cref{alg_3}. In particular, we show that the limit point of the sequence $\{x^k\}_{k \geq 0}$, found by \Cref{alg_3}, is a global minimizer to the DC program \eqref{P} if $\epsilon$ is set to zero. Let us introduce the following quantities:
\begin{align} 
	\label{eq:ak_bk}
		a_k \coloneqq g^{k-1}(x^k) - h(x^k), \quad 	b_k \coloneqq g^k(x^k) - h(x^k) \quad \text{for~}  k\in\{1,2,\ldots\}.  
\end{align}
The following lemma highlights some properties of the functions $g^k$ and the quantities $a_k$ and $b_k$. 
\begin{lemma}
	\label{lemma1}
Assume $\epsilon = 0$ in \Cref{alg_3}. Let $g^k, s^k$ be as in \Cref{alg_3} and $b_k,a_k$ be as in \eqref{eq:ak_bk}.  Then, \begin{enumerate}[(a)]
		\item $g^k (x^k)=g(x^k)=s^k(x^k)$ holds for all $k\in\{1,2,\ldots\}$,
		\item $a_k \leq \min_{x \in X}(g(x)-h(x)) \leq b_k$ holds for all $k\in\{1,2,\ldots\}.$
	\end{enumerate}
\end{lemma}

\begin{proof}
	\begin{enumerate}[(a)]
		\item It is clear as we have $g^k(x^k) \leq g(x^k) = s^k(x^k) \leq \max\{g^{k-1}(x^k), s^k(x^k)\} = g^k(x^k). $
		\item Since $x^k$ is an optimal solution of \eqref{eq:pk}, and $g^{k-1}$ is an underestimator of $g$, we have
		\begin{equation*}
			a_k = g^{k-1}(x^k)-h(x^k) = \min_{x \in X}(g^{k-1}(x)-h(x)) \leq \min_{x \in X}(g(x)-h(x)) \leq g(x^k)-h(x^k) = b_k,
		\end{equation*}
		where the last equality is by (a). 
	\end{enumerate}
\end{proof}

\begin{theorem}\label{alg3conv}
	Assume $\epsilon = 0$ in \Cref{alg_3}. Every limit point of the sequence $\{x^k\}_{k\geq0}$ outputted by \Cref{alg_3} is a global minimizer of \eqref{P}.
\end{theorem}
\begin{proof}
	The compactness of $X$ implies that the limit points of the sequence $\{x^k\}_{k \geq 0}$ exist in $X$. Let $\{x^{k_j}\}_{j \geq 1}$ be a convergent subsequence. With the convention that $s^0(x) = g^0(x)$ (used in the second equality below) and using definition of $s^i$ for $i\geq 1$, we have 
	\begin{align}
		b_{k_{j-1}} &= g^{k_{j-1}}(x^{k_{j-1}}) - h(x^{k_{j-1}}) \notag\\
		&= \max_{0 \leq i \leq k_{j-1}} (s^i(x^{k_{j-1}}))-h(x^{k_{j-1}})\nonumber\\
		&=  \max_{0 \leq i \leq k_{j-1}} (s^i(x^{k_j})+c(x^i)^\T (x^{k_{j-1}}-x^{k_j})) -h(x^{k_{j-1}})\nonumber \\
		&\leq \max_{0 \leq i \leq k_{j-1}} s^i(x^{k_j})+ \max_{0 \leq i \leq k_{j-1}} \|c(x^i)\| \| x^{k_{j-1}}-x^{k_j}\|_* -h(x^{k_{j-1}})\nonumber\\
		&=g^{k_{j-1}}(x^{k_j})-h(x^{k_j}) + \max_{0 \leq i \leq k_{j-1}}\|c(x^i)\| \| x^{k_{j-1}}-x^{k_j}\|_*+h(x^{k_j})-h(x^{k_{j-1}})\nonumber\\
		&=a_{k_j} + \max_{0 \leq i \leq k_{j-1}}\|c(x^i)\| \| x^{k_{j-1}}-x^{k_j}\|_*+h(x^{k_j})-h(x^{k_{j-1}}),\nonumber
	\end{align}
	where the inequality is by the triangle and Hölder inequalities. As $h$ is continuous, \\$\lim_{j\to \infty} \left(\max_{0 \leq i \leq k_{j-1}}\|c(x^i)\| \| x^{k_{j-1}}-x^{k_j}\|_*+h(x^{k_j})-h(x^{k_{j-1}})\right) = 0$. Moreover, by \Cref{lemma1} (b), we have $a_{k_j} \leq \min_{x \in X}(g(x)-h(x)) \leq b_{k_j}$. Hence, we obtain $$\limsup_{j \rightarrow \infty} b_{k_{j-1}} \leq \min_{x \in X}(g(x)-h(x)) \leq \liminf_{j \rightarrow \infty} b_{k_{j}}.$$ This shows that  
	\begin{equation*}
		\lim_{j \to \infty}(g(x^{k_j})-h(x^{k_j})) = \lim_{j \to \infty}(g^{k_j}(x^{k_j})-h(x^{k_j})) = \lim_{j \to \infty}b_{k_j} = \min_{x \in X}(g(x)-h(x)),
	\end{equation*} where we use \Cref{lemma1} (a) in the first equality.
\end{proof}

\begin{corollary}
	\label{finite_alg3}
	Let $x^*$ be the global minimizer of \eqref{P} and $z^* = g(x^*)-h(x^*)$. Then, \Cref{alg_3} stops after finitely many iterations, when $\epsilon$ is set to a positive number, $\Tilde{\epsilon}$.
\end{corollary}
\begin{proof}
	By \Cref{alg3conv}, the sequence $\{x^k\}_{k \geq 0}$ outputted by \Cref{alg_3} converges to $x^*$ if $\epsilon$ is set to zero. Then for any $\Tilde{\epsilon} > 0$, there exists $\Tilde{K} \in \mathbb{N}$ such that $|g(x^k)-h(x^k)-z^*| \leq \Tilde{\epsilon}$ for $k \geq \Tilde{K}$. Note that if the algorithm is run for $\epsilon = \Tilde{\epsilon}$ instead of $\epsilon = 0$, then the first $\Tilde{K}$ iterations would be the same by the structure of the algorithm. In particular, we can assume without loss of generality that the same $(x^k,y^k)$ in line 5 of \Cref{alg_3} is selected for every $k \leq \Tilde{K}$. This implies that \Cref{alg_3} stops in $\Tilde{K}$ iterations when it runs with $\epsilon = \Tilde{\epsilon}$.
\end{proof}

\section{Computational results}\label{sect:ex}
In this section, we solve some test examples from \cite{ferrer2015solving} to assess the performance of the proposed algorithms, which are implemented using MATLAB R2022a along with \textit{bensolve tools} \cite{lohne2017vector} to solve the vertex enumeration problem in each iteration. The tests are run on a computer having a 3.6 GHz Intel Core i7 with 64 GB RAM.

We consider eight examples of the form 
\begin{align}
	\text{minimize} \quad f(x) \quad \text{subject to} \quad  l \leq x \leq u, \nonumber
\end{align}
where $f : \R^n \to \R$ is a DC function written as $f = g - h$ for convex functions $g,h : \R^n \to \R$ and $l,u \in \R^n$. We denote the vector of ones in $\R^n$ by $e$. The examples are listed below.
\begin{enumerate}
	\item \cite[Pr. 10.3]{ferrer2015solving} $n=1, f(x):= -\log(x) + \min\{\sqrt{\abs{1-x}},(2-x)^3,\sqrt{\abs{3-x}}\}$, $l =1, u= 3$. DC components are
	\begin{align}
		g(x) &= 6x^2-12x+8+\max\{0,-x^3\}-\log(x) \coloneqq G(x)-\log(x),\nonumber\\ 
		h(x) & = \max\{-\sqrt{\abs{3-x}} + G(x), -\sqrt{\abs{1-x}}+G(x), \max\{0,x^3\}\}.\nonumber
	\end{align}		
	The problem attains an optimal solution at $x^* = 3$ with minimum value $-1-\log 3$. 
	\item \cite[Pr. 10.1]{ferrer2015solving} $n=2, f(x) := -\sin(\sqrt{3x_1+2x_2+\abs{x_1-x_2}}), l =(0,0)^\T, u= (5,5)^\T$. DC components are
	\begin{equation*}
		g(x) = 5(x_1^2+x_2^2), \quad h(x)  = \sin(\sqrt{3x_1+2x_2+\abs{x_1-x_2}}) +5(x_1^2+x_2^2).			
	\end{equation*}
	\item \cite[Pr. 10.6]{ferrer2015solving} $n=2, f(x) := (x_1^2+ 0.09 x_1)(x_2^2+0.1 x_2), l =(-2,-2)^\T, u= (1,1)^\T$. DC components are
	\begin{equation*}
		g(x) = (x_1^2+0.09x_1)(x_2^2+0.1x_2) + 7.5 (x_1^2+x_2^2), \quad h(x)  = 7.5 (x_1^2+x_2^2).			
	\end{equation*}
	\item \cite[Pr. 10.7]{ferrer2015solving} $n=2, f(x) := \frac{1}{4}(x_1+x_2)^2-\frac{1}{4}(x_1-x_2)^2, l =(-2,-3)^\T, u= (3,4)^\T$. DC components are
	\begin{equation*}
		g(x) = \frac{1}{4}(x_1+x_2)^2, \quad h(x)  = \frac{1}{4}(x_1-x_2)^2.	
	\end{equation*}
	The problem attains an optimal solution at $(3,-3)^\T$ with minimum value -9. 
	\item \cite[Pr. 10.8]{ferrer2015solving} $n=2, f(x) := 0.03(x_1^2+x_2^2)-\cos(x_1)\cos(x_2), l =(-6,-5)^\T, u= (4,2)^\T$. DC components are
	\begin{equation*}
		g(x) = 1.03(x_1^2+x_2^2)-\cos(x_1)\cos(x_2), \quad h(x)  = (x_1^2+x_2^2). 			
	\end{equation*}
	The problem attains an optimal solution at $(0,0)^\T$ with minimum value -1. 
	\item \cite[Pr. 10.5]{ferrer2015solving} $n \in \{2,3\}$, $m \in \{2,3\}$, 
	$f(x) := -\sum_{i=1}^{m}\frac{1}{\norm{x-a_ie}^2+c_i}$, $l = 0 e, u = 10 e$, where $a=(4, 2.5, 7.5)^\T, c=(0.70, 0.73, 0.76)^\T$ are parameters of the problem. 
	DC components are
	\begin{equation*}
		g(x) = f(x)+\norm{x}^2, \quad h(x)  = \norm{x}^2.	
	\end{equation*}
\item \cite[Pr. 10.9]{ferrer2015solving} $n=4$, $ l = -10 e, u=10 e$, $f(x) = g(x)-h(x)$ with \begin{align}
	g(x) &:= \abs{x_1-1} + 200 \max\{0,\abs{x_{1}} - x_2\} + 180 \max\{0,\abs{x_{2}} - x_3\} + \abs{x_1-1} \notag \\ & \quad \quad + 10.1 (\abs{x_2-1}+\abs{x_4-1}) + 4.95 \abs{x_2+x_4-2}, \notag\\ 
	h(x)&:= 100 (\abs{x_{1}} - x_2) + 90 (\abs{x_{3}} - x_4) + 4.95 \abs{x_2-x_4}. \notag
	\end{align} 
\item \cite[Pr. 10.10]{ferrer2015solving} $n \in \{2,3,4,5\}$, $l = -10 e, u=10 e$, $f(x) = g(x)-h(x)$ with \begin{align}
	g(x) &:= \abs{x_1-1} + 200 \sum_{i=2}^n \max\{0,\abs{x_{i-1}} - x_i\}, \quad
	h(x) := 100 \sum_{i=2}^n (\abs{x_{i-1}} - x_i). \notag 
\end{align}		
\end{enumerate}

Example 1 has a univariate objective function ($n=1$); Examples 2-5 have bivariate objective functions ($n=2$); and Example 6 is scalable and is solved for $n \in \{2,3\}$. Examples 7-8 are polyhedral DC programming instances where $n=4$, and $n \in \{2,3,4,5\}$, respectively. 
Epigraphs of $\epsilon$-polyhedral approximations of $g$ returned by Algorithms \ref{alg_1}-\ref{alg_3} for Example 3 are shown in \Cref{fig:ex:alg1} for illustrative purposes. As expected, when the algorithms are run for the same $\epsilon$ value, \Cref{alg_1_mod} returns a much finer approximation of $g$ compared to the others. Moreover, \Cref{alg_3} returns an underestimator of $g$ which approximates $g$ locally around the optimal solution, as expected.  

\begin{figure}[H]
	\centering
	\subfloat[\Cref{alg_1}, $\epsilon=1$]{\includegraphics[width=2in]{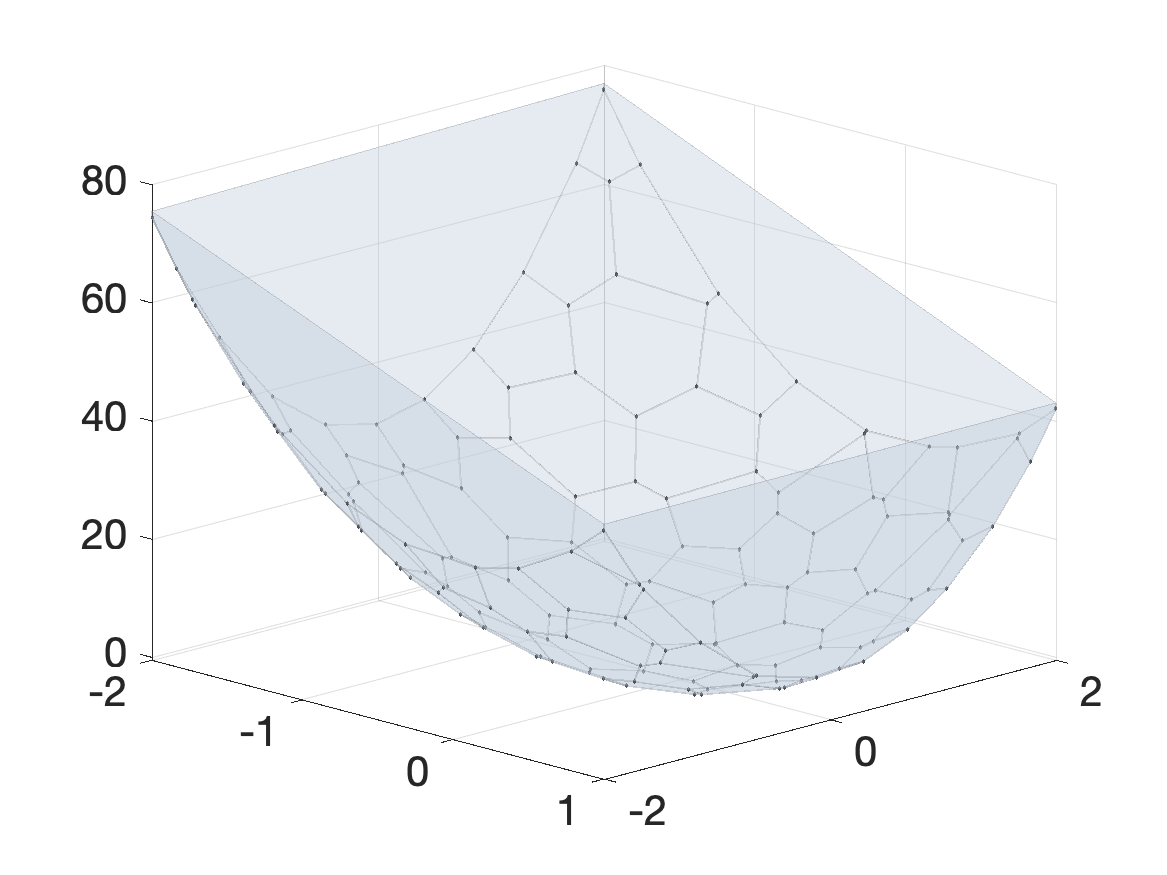}} 
	\subfloat[\Cref{alg_1}, $\epsilon=0.1$]{\includegraphics[width=2in]{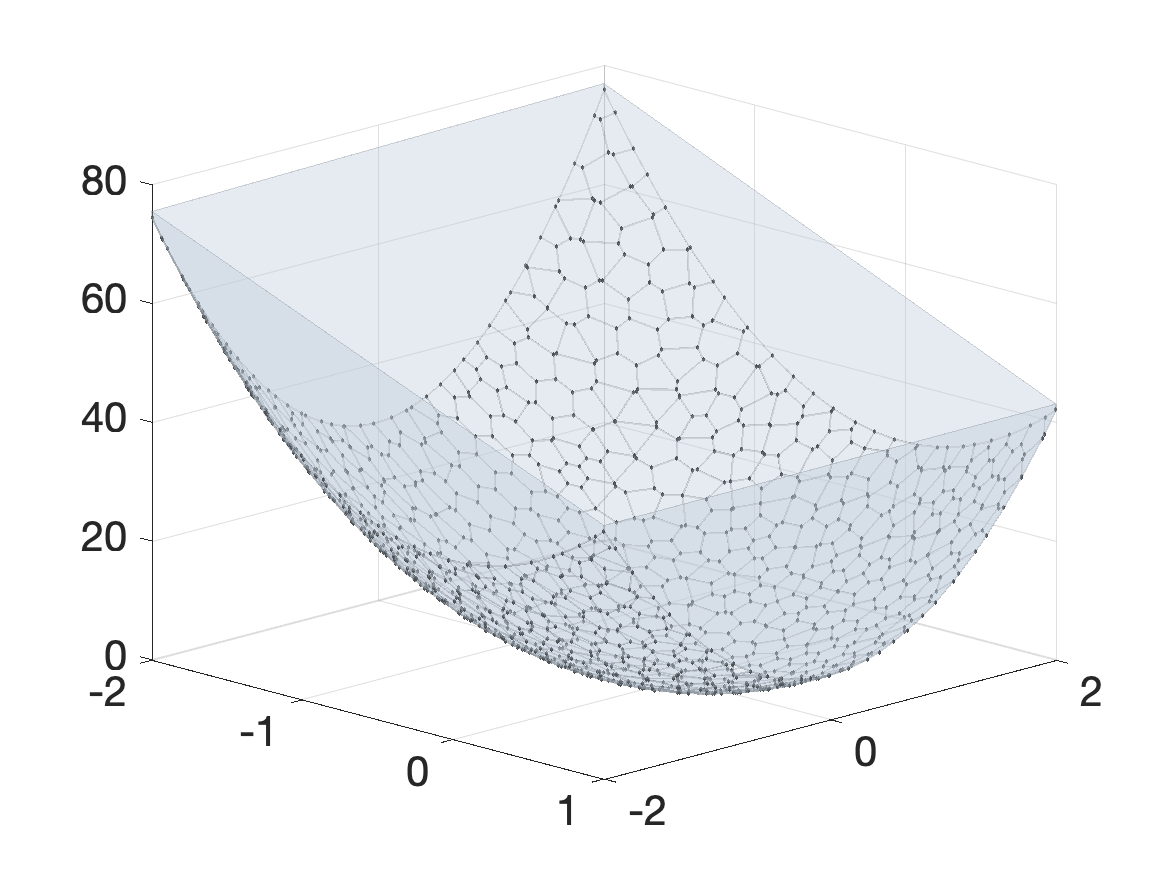}}
	\subfloat[\Cref{alg_1}, $\epsilon=0.1$]{\includegraphics[width=2in]{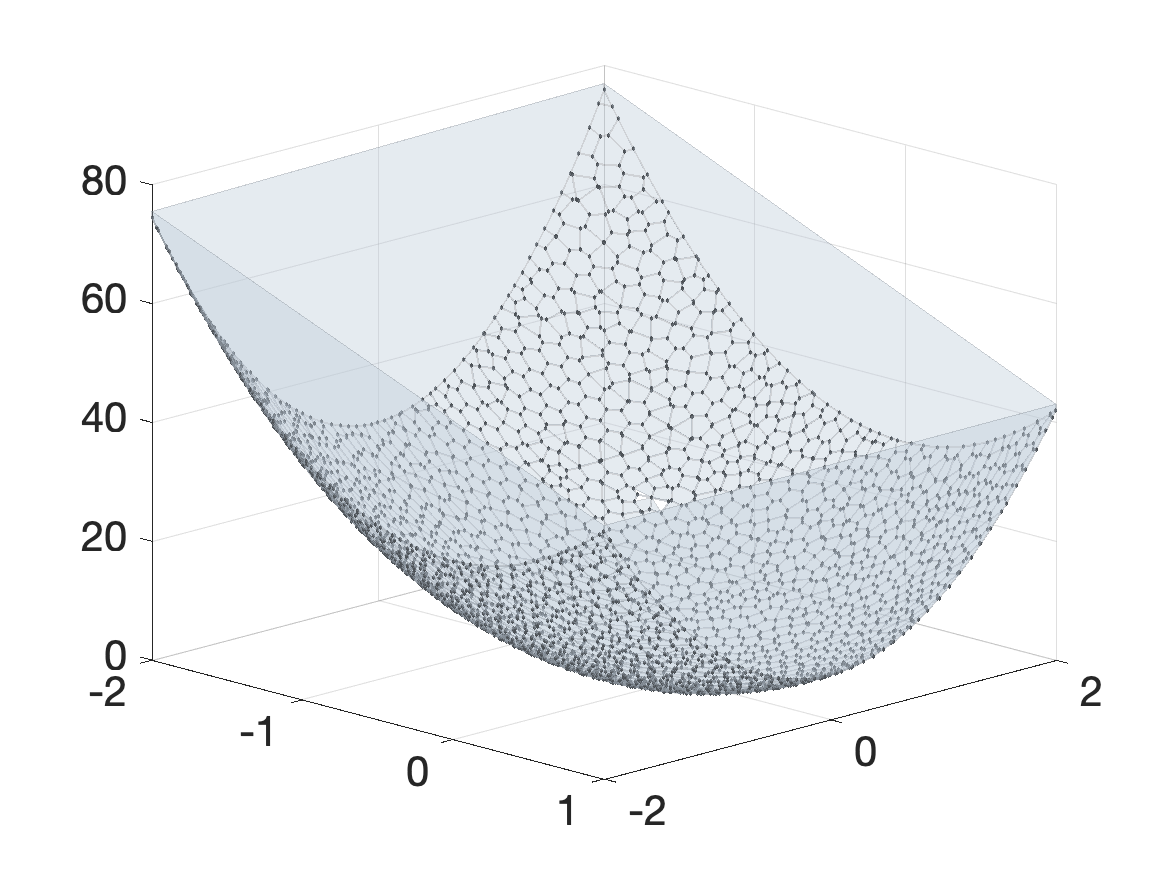}}
	\\
	\subfloat[\Cref{alg_1_mod}, $\epsilon=1$]{\includegraphics[width=2in]{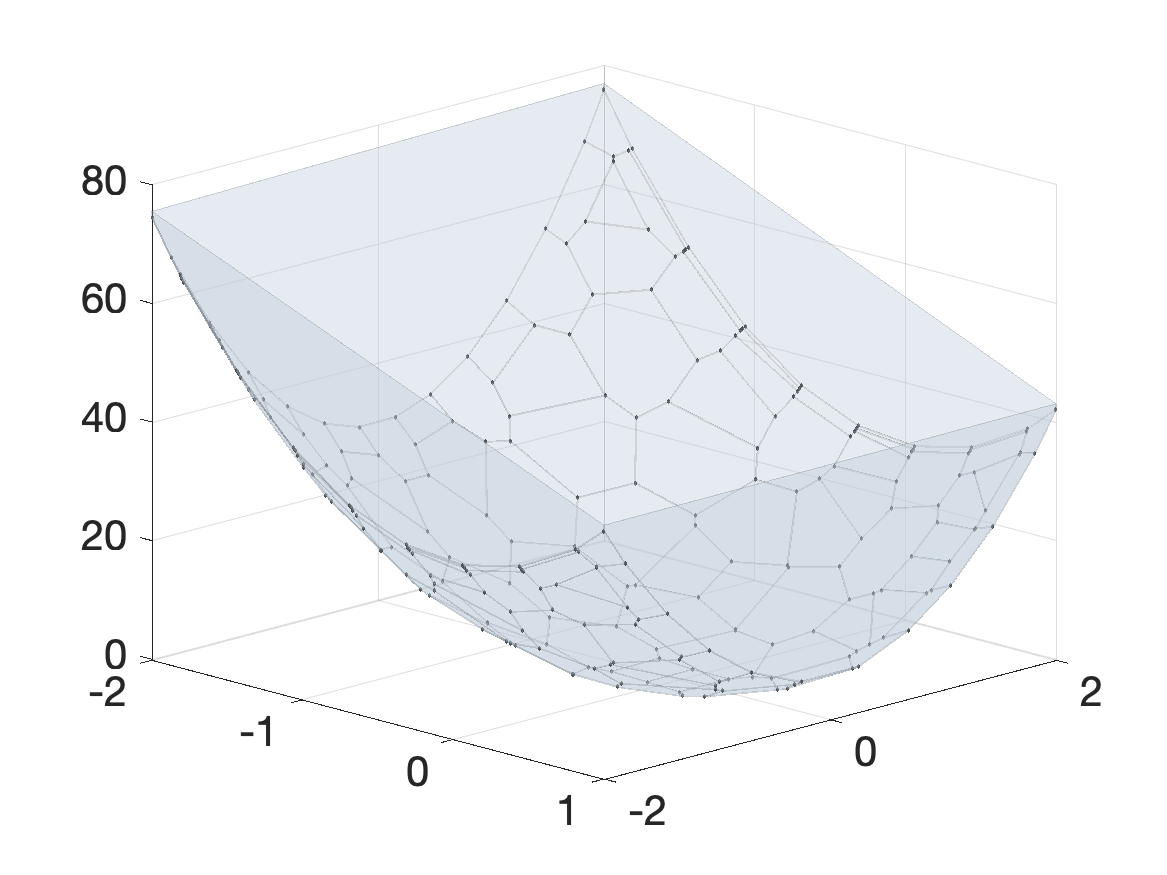}} 
	\subfloat[\Cref{alg_1_mod}, $\epsilon=0.1$]{\includegraphics[width=2in]{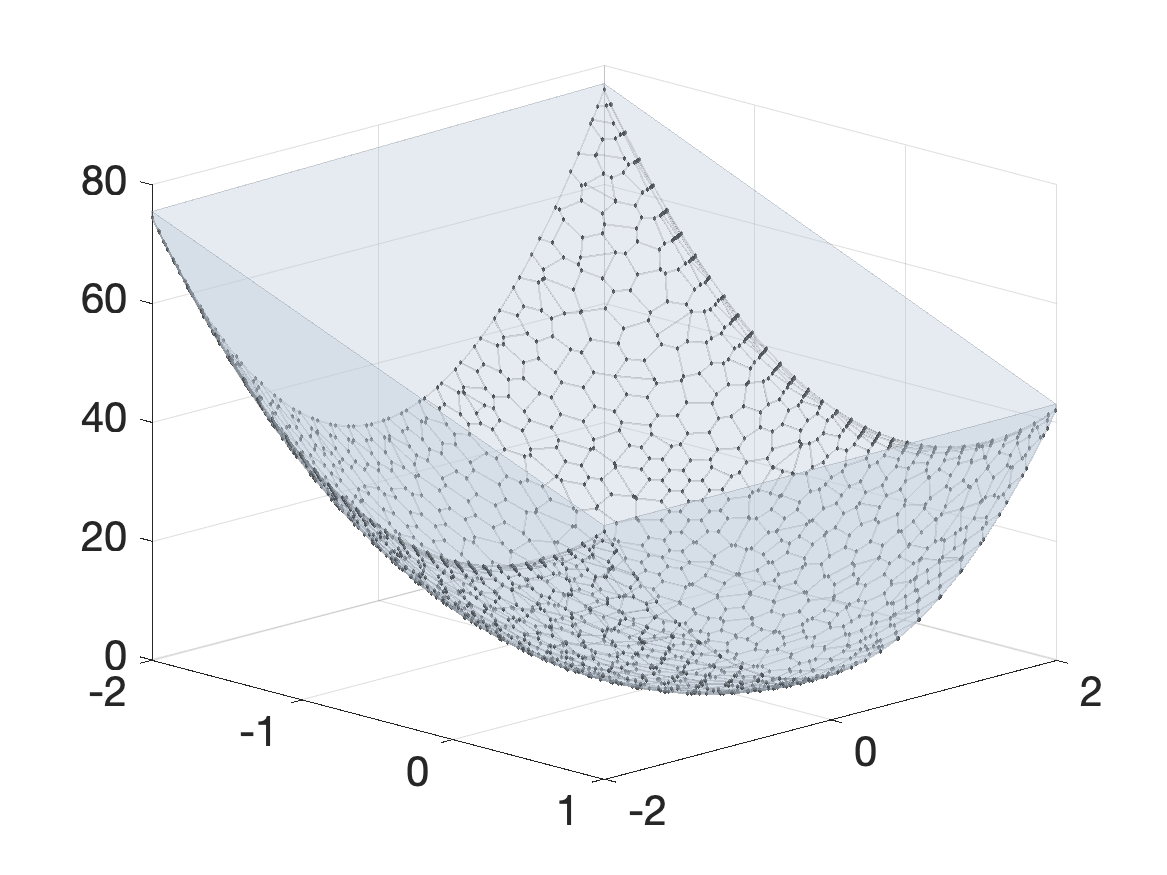}}
	\subfloat[\Cref{alg_1_mod}, $\epsilon=0.1$]{\includegraphics[width=2in]{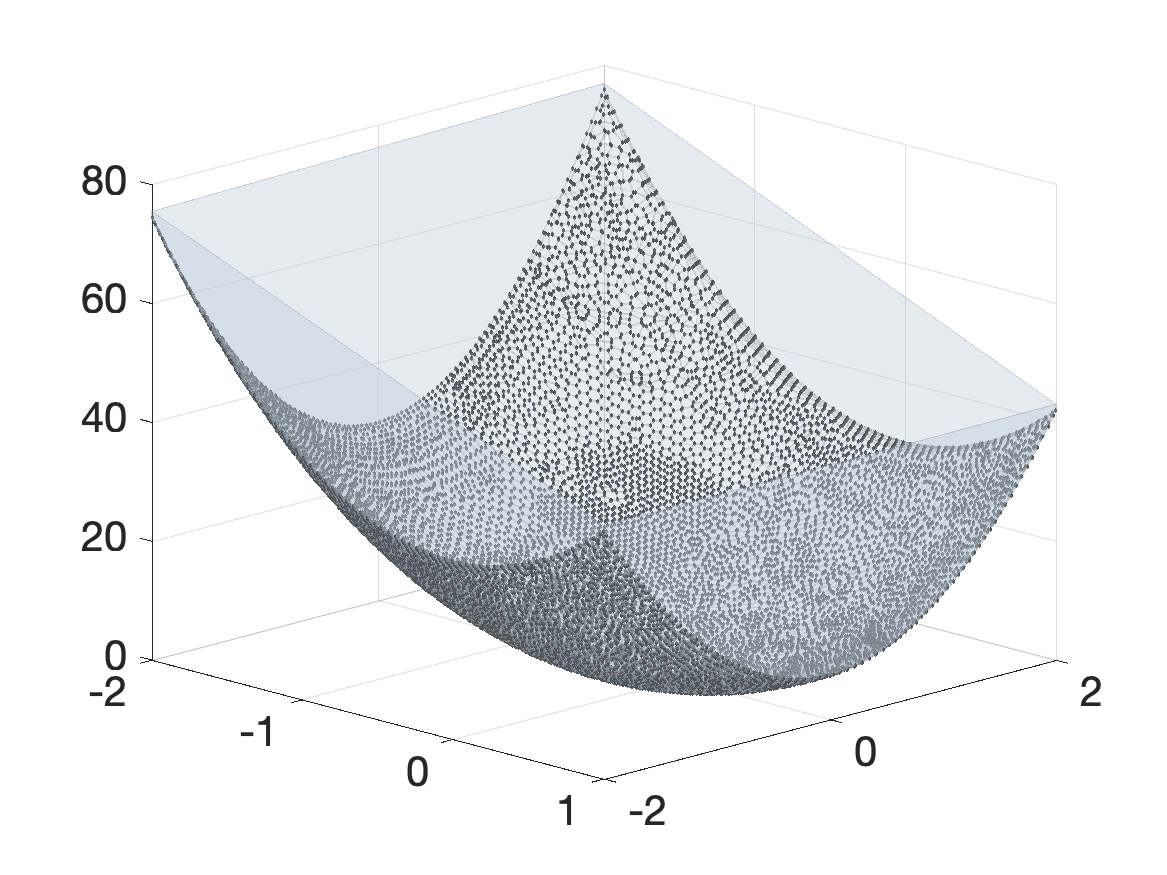}}
	\\
	\subfloat[\Cref{alg_3}, $\epsilon=1$]{\includegraphics[width=2in]{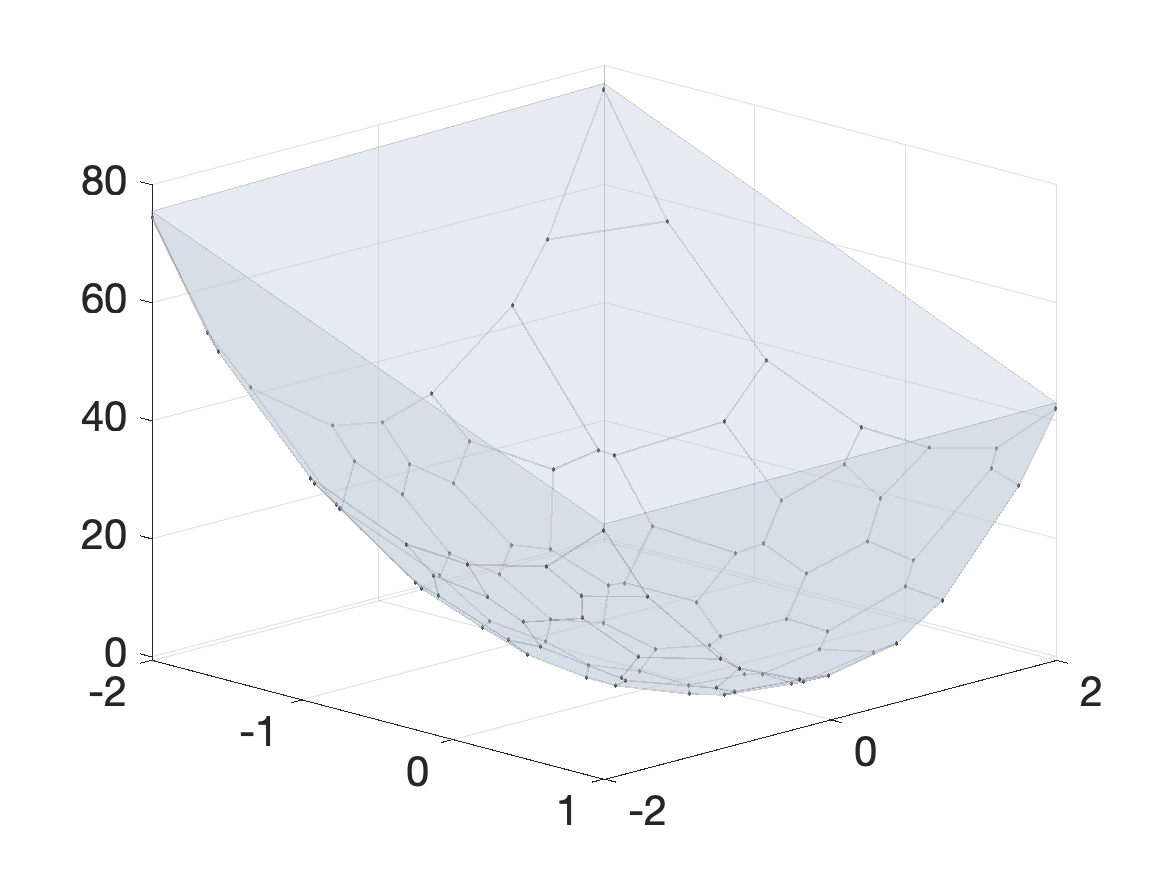}} 
	\subfloat[\Cref{alg_3}, $\epsilon=0.1$]{\includegraphics[width=2in]{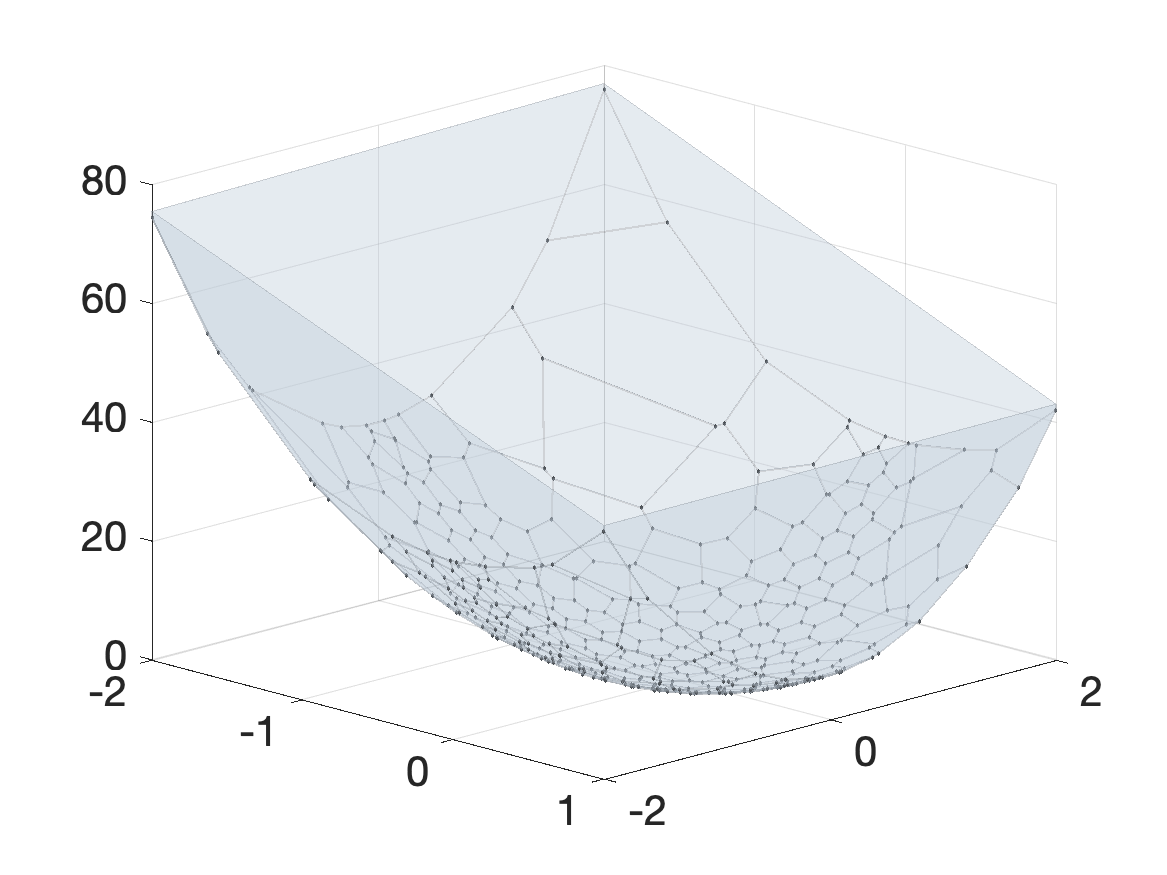}}
	\subfloat[\Cref{alg_3}, $\epsilon=0.1$]{\includegraphics[width=2in]{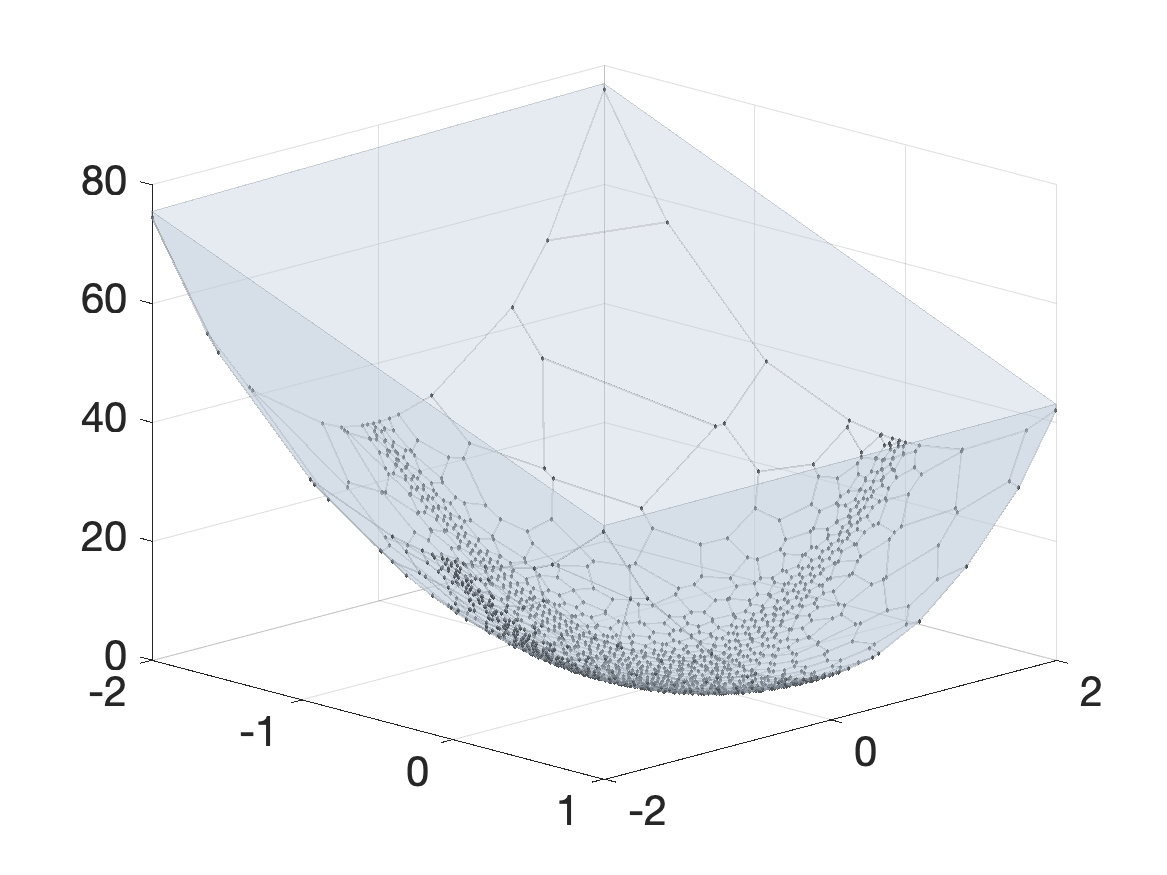}}
	\caption{The polyhedral approximations of $g$ obtained by the algorithms upon termination for Example 3. }
	\label{fig:ex:alg1}
\end{figure}

We solve all examples by Algorithms \ref{alg_1}-\ref{alg_3} for a few different $\epsilon$ values. Tables \ref{tab:1} and \ref{tab:2} show the computational results. In particular, for each example and algorithm, they show the CPU time (time) and the objective function value (value) obtained by the corresponding algorithm. We also write the optimal objective value $z^\ast$ if it is known. For each example, we set a time limit of one hour. If the algorithm hits the time limit, it stops and returns the current solution, which may not be an $\epsilon$-solution. This is indicated by `$> 3600$' in the time column of the tables. 
 
Note that DCECAM proposed in \cite{ferrer2015solving} solves DC programming problems if the first component of the DC function is Lipschitz continuous, and it returns a sequence of solutions that converge to a global optimal solution. However, when stopped using a tolerance $\epsilon > 0$, it doesn't guarantee to return an $\epsilon$-solution in the sense of \Cref{defn:eps}. For each example, we also provide the results regarding DCECAM from \cite{ferrer2015solving} if available.\footnote{There are also some results in \cite[Table 7]{ferrer2015solving} for Example 6. However, there are different choices of parameters for this set of examples and the corresponding table from \cite{ferrer2015solving} does not provide the selected parameter. For the other examples, the values returned by DCECAM are taken as they appear in \cite{ferrer2015solving}. The value returned by DCECAM for Example 3 is smaller than the optimal objective function value as appeared in \cite{ferrer2015solving}, hence not reported here.}

\begin{table}[htbp]  
	\centering  
	\caption{Computational results for Examples 1-6 }
	\resizebox{0.9\textwidth}{!}{ \begin{tabular}{|c|r|r|c|rr|rr|rr|rr|}    \hline   \multicolumn{1}{|c|}{\multirow{2}[2]{*}{Ex}} & \multicolumn{1}{c|}{\multirow{2}[2]{*}{$n$}} & \multicolumn{1}{c|}{\multirow{2}[2]{*}{$z^\ast$}} &       & \multicolumn{2}{c|}{Alg 1} & \multicolumn{2}{c|}{Alg 2} & \multicolumn{2}{c|}{Alg 3} & \multicolumn{2}{c|}{DCECAM} \\          &       &       & eps   & \multicolumn{1}{c}{time} & \multicolumn{1}{c|}{value} & \multicolumn{1}{c}{time} & \multicolumn{1}{c|}{value} & \multicolumn{1}{c}{time} & \multicolumn{1}{c|}{value} & \multicolumn{1}{c}{time} & \multicolumn{1}{c|}{value} \\    \hline   \multicolumn{1}{|c|}{\multirow{3}[2]{*}{1}} & \multicolumn{1}{c|}{\multirow{3}[2]{*}{1}} & \multicolumn{1}{c|}{\multirow{3}[2]{*}{-1-log 3 }} & 1     & 0.0156 & -2.0986 & 0.0156 & -2.0986 & 0.0156 & -2.0986 & \multirow{3}[2]{*}{0.3100} & \multirow{3}[2]{*}{-2.0927} \\          &       &       & 0.1   & 0.0781 & -2.0986 & 0.0156 & -2.0986 & 0.0781 & -2.0986 &       &  \\          &       &    & 0.01  & 0.2188 & -2.0986 & 0.0313 & -2.0986 & 0.0156 & -2.0986 &       &  \\    \hline    \multicolumn{1}{|c|}{\multirow{3}[2]{*}{2}} & \multicolumn{1}{c|}{\multirow{3}[2]{*}{2}} & \multicolumn{1}{c|}{\multirow{3}[2]{*}{-1}} & 1     & 2.3594 & -0.9602 & 0.1094 & -0.9602 & 0.7969 & -0.9602 & \multirow{3}[2]{*}{2.6100} & \multirow{3}[2]{*}{1} \\          &       &       & 0.1   & 529.5312 & -0.9999 & 4.4375 & -0.9999 & 3.2031 & -0.9999 &       &  \\          &       &       & 0.01  & $>$ 3600 & -0.9999 & 516.7031 & -1 & 11.9219 & -1 &       &  \\    \hline    \multicolumn{1}{|c|}{\multirow{3}[2]{*}{3}} & \multicolumn{1}{c|}{\multirow{3}[2]{*}{2}} & \multicolumn{1}{c|}{\multirow{3}[2]{*}{-0.00955}} & 1     & 1.6094 & 0.0002 & 0.1875 & 0.0002 & 0.5625 & 0.0099 & \multirow{3}[2]{*}{-} & \multirow{3}[2]{*}{-} \\          &       &       & 0.1   & 178.6094 & -0.0057 & 2.7344 & 0.0003 & 15.6406 & -0.0075 &       &  \\          &       &       & 0.01  & $>$ 3600 & -0.0092 & 460.9688 & -0.0073 & 335.0156 & -0.0091 &       &  \\    \hline    \multicolumn{1}{|c|}{\multirow{3}[2]{*}{4}} & \multicolumn{1}{c|}{\multirow{3}[2]{*}{2}} & \multicolumn{1}{c|}{\multirow{3}[2]{*}{-9}} & 1     & 0.0313 & -9.0000 & 0.0494 & -9 & 0.0313 & -9 & \multirow{3}[2]{*}{0.0500} & \multirow{3}[2]{*}{-9} \\          &       &       & 0.1   & 0.1719 & -9 & 0.0313 & -9 & 0.0313 & -9 &       &  \\          &       &       & 0.01  & 0.2969 & -9 & 0.1250 & -9 & 0.0313 & -9 &       &  \\    \hline    \multicolumn{1}{|c|}{\multirow{3}[1]{*}{5}} & \multicolumn{1}{c|}{\multirow{3}[1]{*}{2}} & \multicolumn{1}{c|}{\multirow{3}[1]{*}{-1}} & 1     & 0.9375 & -0.8659 & 0.0625 & -0.9988 & 0.3594 & -0.9918 & \multirow{3}[1]{*}{1.1400} & \multirow{3}[1]{*}{-0.9995} \\          &       &       & 0.1   & 77.4844 & -0.9979 & 1.5156 & -0.9994 & 0.6719 & -0.9932 &       &  \\          &       &       & 0.01  & $>$ 3600 & -0.9987 & 113.1719 & -0.9995 & 0.9219 & -0.9989 &       &  \\   \hline \multicolumn{1}{|r|}{\multirow{3}[1]{*}{6 (m=2)}} & \multicolumn{1}{r|}{\multirow{6}[3]{*}{2}} & \multirow{3}[1]{*}{} & 1     & 1.5625 & -1.4893 & 0.2188 & -1.3074 & 0.7031 & -1.3260 & \multirow{3}[1]{*}{-} & \multirow{3}[1]{*}{-} \\          &       &       & 0.1   & 167.1406 & -1.6220 & 4.5469 & -1.6151 & 1.0938 & -1.6116 &       &  \\          &       &       & 0.01  & $>$ 3600 & -1.6185 & 3349.6000 & -1.6214 & 2.0625 & -1.6196 &       &  \\\cline{1-1}\cline{3-12}    \multicolumn{1}{|r|}{\multirow{3}[2]{*}{6 (m=3)}} &       & \multirow{3}[2]{*}{} & 1     & 1.8125 & -1.5382 & 0.3281 & -1.3544 & 1.1563 & -1.3431 & \multirow{3}[2]{*}{-} & \multirow{3}[2]{*}{-} \\          &       &       & 0.1   & 183.5156 & -1.6486 & 4.5625 & -1.6576 & 1.0469 & -1.6509 &       &  \\          &       &       & 0.01  & $>$ 3600 & -1.6613 & 3448.7000 & -1.6616 & 1.5781 & -1.6605 &       &  \\    \hline    \multicolumn{1}{|r|}{\multirow{3}[2]{*}{6 (m=2)}} & \multicolumn{1}{r|}{\multirow{6}[4]{*}{3}} & \multirow{3}[2]{*}{} & 1     & 3191.5000 & -1.4111 & 164.8594 & -1.5008 & 133.6094 & -1.3882 & \multirow{3}[2]{*}{-} & \multirow{3}[2]{*}{-} \\          &       &       & 0.1   & $>$ 3600 & -1.4571 & -     & -     & 663.9531 & -1.5442 &       &  \\          &       &       & 0.01  & $>$ 3600 & -1.4571 & -     & -     & 835.0781 & -1.5616 &       &  \\
\cline{1-1}\cline{3-12}    \multicolumn{1}{|r|}{\multirow{3}[2]{*}{6 (m=3)}} &       & \multirow{3}[2]{*}{} & 1     & 3041.9000 & -1.4884 & 153.8281 & -1.5133 & 230.5625 & -1.2927 & \multirow{3}[2]{*}{-} & \multirow{3}[2]{*}{-} \\          &       &       & 0.1   & $>$ 3600 & -1.4884 & -     & -     & 761.8438 & -1.5877 &       &  \\          &       &       & 0.01  & $>$ 3600 & -1.4884 & -     & -     & 889.0781 & -1.5878 &       &  \\    \hline   \end{tabular}}%
\label{tab:1}%
\end{table}%

From \Cref{tab:1}, we observe for each algorithm that the runtime increases for decreased values of $\epsilon$, as expected. Moreover, for Examples 1-6, \Cref{alg_3} performs faster than the others in all cases and the difference is more significant for most of the examples for decreased values of $\epsilon$. When we compare the runtimes of Algorithms \ref{alg_1} and \ref{alg_1_mod}, we see that \Cref{alg_1_mod} excels \Cref{alg_1} in most cases. The only exception is Example 6 with $n=3$ and $\epsilon \in \{0.1,0.01\}$, in which none of the two algorithms stop based on the original stopping criteria. \Cref{alg_1} returns a solution after hitting the runtime limit, whereas \Cref{alg_1_mod} does not return a solution even after the runtime limit since \emph{bensolve tools} crushes while intersecting more than $40000$ halfspaces in one iteration. We see that \Cref{alg_3} is comparable to the DCECAM based on the available results in terms of the runtimes and the returned values for Examples 1-6.   

When we compare the objective function values returned by each algorithm for different values of $\epsilon$ in \Cref{tab:1}, we see that for some examples (for instance Ex 1, 2, and 4) Algorithms \ref{alg_1} and \ref{alg_1_mod} require much higher runtimes for smaller $\epsilon$ values even though the objective function value does not improve (much). The reason is Algorithms \ref{alg_1} and \ref{alg_1_mod} run until finding an $\epsilon$-polyhedral underestimator of $g$ without checking any optimality condition for \eqref{P}. This clearly is not the case for \Cref{alg_3}.  

\begin{table}[htbp]  
	\centering  
	\caption{Computational results for Examples 7-8: Algorithms 1-3 are run with $\epsilon = 1$ and return $x^\ast = e \in \R^n$ and the optimal objective function value $z^\ast = 0$ for each example and $n$.}    
	\resizebox{0.5\textwidth}{!}{
	\begin{tabular}{|c|r|r|r|r|rr|}    \hline    \multirow{2}[2]{*}{Ex} & \multicolumn{1}{r|}{\multirow{2}[2]{*}{$n$}} & \multicolumn{1}{c|}{Alg 1} & \multicolumn{1}{c|}{Alg 2} & \multicolumn{1}{c|}{Alg 3} & \multicolumn{2}{c|}{DCECAM} \\          &       & \multicolumn{1}{c|}{time} & \multicolumn{1}{c|}{time} & \multicolumn{1}{c|}{time} & \multicolumn{1}{c}{time} & \multicolumn{1}{c|}{value} \\    \hline    7     & 4     & 231.1094 & 6.9531 & 7.4063 & 9.7100 & 0.0024 \\    \hline    \multirow{4}[2]{*}{8} & 2     & 0.0313 & 0.0313 & 0.0313 & 0.21  & 0 \\          & 3     & 0.2656 & 0.0469 & 0.1094 & 3.57  & 3.572103 \\          & 4     &    9.4844   & 0.6406 & 2.5313 & 2.74  & 0.553429 \\          & 5     &    445.4844   & 21.5469 & 195.1563 & 345.12 & 1.500652 \\    \hline    
\end{tabular}} 
	\label{tab:2}
\end{table}

In \Cref{tab:2}, the $\epsilon$ values as well as the objective function values returned by Algorithms \ref{alg_1}-\ref{alg_3} are not reported. We run all three algorithms for $\epsilon \in \{1,0.1,0.01\}$ as in the previous set of examples. However, for Examples 6 and 7, all three algorithms return the optimal solution $x^\ast = e \in \R^n$ with optimal objective function value zero when run under all $\epsilon$ values. Moreover, the runtimes do not increase by the decreased values of $\epsilon$ for these examples. The reason may be the fact that function $g$ is polyhedral convex, hence can be computed exactly. In \Cref{tab:2}, we only report the results for $\epsilon = 1$.

For Examples 6 and 7, \Cref{alg_1_mod} outperforms the others in terms of the runtime and the difference is notable, especially for $n=5$. On the other hand, \Cref{alg_3} is significantly faster than \Cref{alg_1} and it is also faster than DCECAM in all instances. A main difference between DCECAM and the proposed algorithms is seen in the returned objective function values as the optimality gap returned by DCECAM is quite high for some instances, see for instance, Example 7 with $n=3$.    

Overall, we observe that \Cref{alg_3} has consistently better performance than DCECAM based on the available instances and data from \cite{ferrer2015solving}. On the other hand, for the polyhedral DC instances tested for this study, \Cref{alg_1_mod} performs better than \Cref{alg_3}. However, it is significantly worse than \Cref{alg_3} in other (non-polyhedral) instances especially as $\epsilon$ decreases. We also observe that intersecting the current approximation with more halfspaces (\Cref{alg_1_mod}) than a single halfspace (\Cref{alg_1}) in a single iteration when finding a polyhedral $\epsilon$-underestimator of a (non-polyhedral) convex function reduces the computational time, significantly. It is also worth noting that if the number of halfspaces to intersect at once is significantly high (more than 40000 in our test instances), then there is a risk of encountering technical/numerical issues in \emph{bensolve tools}. In that sense, \Cref{alg_1} still has an advantage compared to \Cref{alg_1_mod}.

\section{Conclusion} \label{conc}
In this paper, we consider DC programming problems and propose global approximation algorithms. First, we propose two algorithms to approximate a convex function over a box by iteratively generating a polyhedral underestimator of it via its affine minorants. Then, the polyhedral underestimator of the first convex component of a DC function obtained by these algorithms is used to solve the corresponding DC programming problem. We prove that both algorithms work correctly. We establish the convergence rate of \Cref{alg_1} and prove the finiteness of \Cref{alg_1_mod}. 

We propose another algorithm (\Cref{alg_3}) which also iteratively generates polyhedral underestimators of the first component $g$ of the DC function. Different from the others, it keeps updating the polyhedral underestimator of $g$ locally while searching for an $\epsilon$-solution of the DC programming problem directly. We prove the correctness and finiteness of \Cref{alg_3}. Moreover, we show that the sequence $\{x^k\}_{k \geq 0}$, outputted by \Cref{alg_3} converges to a global minimizer of the DC programming problem. Computational results show the satisfactory behavior of our proposed algorithms. 

As a future research direction, the convergence rate of \Cref{alg_1_mod} could be established using similar means used in the case of \Cref{alg_1}. The challenge is that, as per the definition of $H$-sequence of outer approximating polytopes, only a single halfspace is intersected at each iteration to update the current approximation. Hence, the results from \cite{lotov2004interactive} cannot be applied directly. 

Finally, the method for solving polyhedral DC programs for DC functions with the \emph{second} component being polyhedral convex, from \cite{lohne2017solving}, could be integrated with our approach to propose further approximation algorithms to solve DC programming problems globally. 

\section*{Declarations}
This manuscript has no associated data. 

\bibliographystyle{plain}
\bibliography{ref}

\newpage
\end{document}